\documentclass[11pt]{amsart}
\usepackage[color=blue!20!white,textsize=tiny]{todonotes}

\usepackage{bm}
\usepackage{multicol}

\usepackage{mathrsfs}
\usepackage{amsmath} 
\usepackage{amsthm}
\usepackage{amssymb} 
\usepackage{amsfonts}
\usepackage{tikz-cd}
\usepackage{graphicx}
\usepackage{xcolor}
\definecolor{darkgreen}{rgb}{0,0.5,0}
\usepackage[normalem]{ulem}
\usepackage{comment}

\usepackage{caption,subcaption,pinlabel}
\usepackage{graphics}
\usepackage[pagebackref=true]{hyperref}
\usepackage{scalerel}
\usepackage{comment}
\usepackage{enumitem}
\usepackage{mathtools}
\usepackage{tikz}
\usetikzlibrary{matrix,arrows,decorations.pathmorphing}
\usepackage{mathabx}
\usetikzlibrary{matrix}
\usepackage{scrextend}

\newtheorem{theorem}{Theorem}[section]
\newtheorem{lemma}[theorem]{Lemma}

\newtheorem{corollary}[theorem]{Corollary}

\theoremstyle{definition}
\newtheorem{definition}[theorem]{Definition}
\newtheorem{example}[theorem]{Example}

\newtheorem{remark}[theorem]{Remark}
\newtheorem{conjecture}[theorem]{Conjecture}

\newcommand{\inv}{\iota}

\def\Inv{\mathfrak{I}}

\def\spinc {{\operatorname{spin^c}}}
\def\s{\mathfrak s}

\def\x{\mathbf{x}}
\def\CF {\mathit{CF}}
\def\HF {\mathit{HF}}

\newcommand \CFm {\CF^-}

\newcommand \Hc {\HF_{\mathrm{conn}}}

\def\ff {{\mathbb{F}}}

\newcommand{\N}{\mathbb{N}}
\newcommand{\Z}{\mathbb{Z}}
\newcommand{\Q}{\mathbb{Q}}

\newcommand{\F}{\mathbb{F}}

\let\int\relax
\newcommand{\int}{\mathring}

\usepackage{accents}

\DeclareMathSymbol{\wtilde}{\mathord}{largesymbols}{"65}

\mathchardef\mhyphen="2D

\newcommand{\CFK}{\mathit{CFK}}

\newcommand{\gr}{\text{gr}}

\newcommand{\id}{\mathrm{id}}

\newcommand{\Wh}{D}
\newcommand{\Wht}{D_{t}}

\newcommand{\ima}{\mathrm{im}}

\newcommand{\cC}{\mathcal{C}}

\newcommand{\tmu}{\smash{\widetilde{\mu}}}
\newcommand{\lk}{\mathit{lk}}
\title{Rank-expanding satellites, Whitehead doubles, and Heegaard Floer homology}

\author[I. Dai]{Irving Dai}
\address {Department of Mathematics, Stanford University, Palo Alto, CA 94301}
\email{ifdai@stanford.edu}

\author[M. Hedden]{Matthew Hedden}
\address {Department of Mathematics, Michigan State University, East Lansing, MI 48824}
\email{heddenma@msu.edu}

\author[A. Mallick]{Abhishek Mallick}
\address{Simons Laufer Mathematical Sciences Institute (previously MSRI), Berkeley}
\email{abhishek.mallick@rutgers.edu}

\author[M. Stoffregen]{Matthew Stoffregen}
\address {Department of Mathematics, Michigan State University, East Lansing, MI 48824}
\email{stoffre1@msu.edu}

\begin{document}
\vspace*{-1cm}
\maketitle
\vspace*{-0.4cm}
\begin{abstract}
We show that a large class of satellite operators are rank-expanding; that is, they map some rank-one subgroup of the concordance group onto an infinite linearly independent set. Our work constitutes the first systematic study of this property in the literature and partially affirms a conjecture of the second author and Pinz\'on-Caicedo. More generally, we establish a Floer-theoretic condition for a family of companion knots to have infinite-rank image under satellites from this class. The methods we use are amenable to patterns which act trivially in topological concordance and are capable of handling a surprisingly wide variety of companions. For instance, we give an infinite linearly independent family of Whitehead doubles whose companion knots all have negative $\tau$-invariant. Our also results recover and extend several theorems in this area established using instanton Floer homology. 
\end{abstract}

\section{Introduction}\label{sec:1}
For any pattern knot $P \subset S^1 \times D^{2}$, the  satellite operation $K \mapsto P(K)$ induces a map 
\[
P: \mathcal{C} \rightarrow \mathcal{C}
\]
on the smooth (or topological) knot concordance group.   These operators have been central to the study of the concordance groups in both categories; for example, see \cite{cochranteichner,cochranorr,livingston-boundary,fractal,2torsionsolvable,HKL,HLR,CochranHarveyHorn,grope,Hominfiniterank,Levinenonsurj,injectivity,amphicheiral,reverses,MR780587,davisray,Chen11patterns,HK,PC,daemi2020chern,HPC,NST,MillerHomo,lidman2022linking,DISST}.  In this article, we investigate several questions regarding the rank of different satellite operators on the smooth concordance group.  The starting point for this line of research is the following conjecture, due to the second author and Pinz\'on-Caicedo \cite{HPC}:

\begin{conjecture}\cite[Conjecture 2]{HPC}\label{conj:1.1}
Every non-constant satellite operator has infinite rank.
\end{conjecture} 

Here, by the rank of $P$ we mean the rank of the subgroup generated by the image of $P$, since in general $P$ is not a homomorphism. Significant progress towards Conjecture~\ref{conj:1.1} was made in \cite[Theorem 3]{HPC}, where it was verified for all winding number zero patterns satisfying a certain rational linking number condition.\footnote{The proof in the case of non-zero winding number is straightforward and follows from a consideration of Tristram-Levine signatures; see \cite[Proposition 8]{HPC}.} Specifically, it was shown that any such pattern maps a carefully-selected sequence of torus knots to an infinite linearly independent set. Other research has focused on establishing the linear independence of explicit families of knots under patterns such as Whitehead doubling. For instance, in joint work with Kirk, the second author proved the Whitehead doubles $\{\Wh(T_{2, 2^k-1})\}_{k \geq 2}$ are linearly independent \cite[Theorem 1]{HK}; this was extended to the entire family $\{\Wh(T_{2, 2k+1})\}_{k \in \N}$ by Nozaki, Sato, and Taniguchi \cite[Corollary 1.13]{NST}. (See also \cite[Theorem 1.12]{NST}.) Linear independence of torus knots under other (Whitehead-like) satellites was studied by Pinz\'on-Caicedo in \cite{PC}.

Given the linear independence of torus knots in $\cC$, the above results should be thought of as examples of rank-preserving behavior for $P$. The existence of more exotic behavior was conjectured in \cite{HPC}, where the following strengthening of Conjecture~\ref{conj:1.1} was presented:

\begin{conjecture}\cite[Conjecture 4]{HPC}\label{conj:1.2} 
If $P$ is a non-constant winding number zero satellite operator, then there exists a knot $K$ for which $\{P(nK)\}_{n \in \Z}$ has infinite rank.
\end{conjecture} 

By the rank of $\{P(nK)\}_{n \in \Z}$, we again mean the rank of the subgroup generated by $\{P(nK)\}_{n \in \Z}$. Conjecture~\ref{conj:1.2} states that any nontrivial satellite operator sends some rank-one subgroup of $\cC$ surjectively onto an infinite linearly independent set. We formalize this in the following definition:

\begin{definition}\label{def:1.3}
A satellite operator $P$ is \textit{rank-expanding} if there exists a rank-one subgroup $\{nK\}_{n \in \Z}$ of $\cC$ such that $\{P(nK)\}_{n \in \Z}$ has infinite rank. When we wish to emphasize the knot $K$, we say that $P$ is rank-expanding \textit{along $\{nK\}_{n \in \Z}$} (or sometimes just \textit{along $K$}).\footnote{Note that implicitly, $K$ is required to be nontorsion in $\cC$. One can also define rank expansion by requiring that there is some finite-rank subgroup whose image under $P$ generates a subgroup of greater (but still possibly finite) rank; here, we have instead chosen the strongest possible notion. The authors briefly considered calling the operators of Definition~\ref{def:1.3} rank-\textit{exploding}.}
\end{definition}

Prior to the current article, little was known about Conjecture~\ref{conj:1.2}, even in specific cases. Indeed, in \cite{HPC} it was asked whether $\{P(nK)\}_{n \in \N}$ is linearly independent for $P$ the Whitehead double and $K$ the trefoil. We show:

\begin{corollary}\label{cor:A}
Let $\mathcal{F}$ be any subset of $\{\Wh(nT_{2, 2k+1})\}_{n, k \in \N}$ whose index pairs have distinct products $nk$. Then $\mathcal{F}$ is linearly independent and in fact spans a $\Z^\infty$-summand of $\cC$.
\end{corollary}
\noindent
Setting $k = 1$ and varying $n$ yields the family $\mathcal{F} = \{n\Wh(T_{2,3})\}_{n \in \N}$. This answers the above question in the affirmative and (in particular) shows that $D$ is rank-expanding along $T_{2,3}$. Setting $n = 1$ and varying $k$ yields the family $\mathcal{F} = \{\Wh(T_{2, 2k+1})\}_{k \in \N}$, which recovers \cite[Corollary 1.13]{NST} (and thus \cite[Theorem 1]{HK}). Corollary~\ref{cor:A} is a consequence of a much more general result and can be extended to all multiply-clasped and twisted Whitehead doubling operators; see Theorem~\ref{thm:1.10}.

In fact, we verify Conjecture~\ref{conj:1.2} for many other patterns and families of companions in Theorem~\ref{thm:1.7}. The prevailing belief seems to be that nontrivial satellite operators are never homomorphisms (for example, see \cite{MillerHomo, Chen11patterns, lidman2022linking}). Our results indicate that they are, in some quantifiable sense, maximally far from being homomorphisms. Indeed, a potentially reasonable strengthening of Conjecture~\ref{conj:1.2} would be the following: 
\begin{conjecture}
Any non-constant winding number zero satellite operator is rank-expanding along \textit{every} rank-one subgroup $\{nK\}_{n \in \Z}$.
\end{conjecture}
\noindent
It is thus natural to establish robust conditions which affirm rank expansion along different $K$.

The difficulty with studying the rank of satellite operators, especially for patterns such as Whitehead doubling, lies principally with a lack of effective invariants. For example, as discussed in \cite[Section 1]{HPC}, the knot Floer homology of Whitehead doubles is sufficiently constrained so that the usual suite of Floer-theoretic concordance invariants (such as $\tau$, $\Upsilon$, stable equivalence, and so on) cannot be used to establish linear independence.  The most common technique to date has been to pass to the branched double covers of these knots and utilize  homology cobordism invariants of the latter manifolds.  In the case that these manifolds have non-trivial first homology (when the determinant of the satellite knots is not one), there are a host of Fr{\o}shov-type invariants coming out of Floer theories, or analogous Casson-Gordon signatures available in the topological category.  Analyzing these invariants in conjunction with metabolizers for linking forms yields a powerful tool for studying satellite operators, and can  be used to show that certain operators whose image consists of satellite knots with non-zero determinant have infinite rank, and are even rank expanding. (For example, Chuck Livingston pointed out to the authors that Casson-Gordon invariants can verify that certain {\em twisted} Whitehead doubles are rank-expanding. See \cite{HLR,HKL,CochranHorn,CochranHarveyHorn} for related results using $d$-invariants.)  

However, when the determinant of the satellite knots is one, the branched double covers are homology spheres, and such techniques break down.  To date, the only method for bypassing this has been to employ the filtration on instanton Floer homology provided by the Chern-Simons functional.  Since instanton Floer homology is only well-understood for a small subset of 3-manifolds, this approach has only been used to study the images of very restricted families of companion knots, such as those closely related to torus knots \cite{HK, PC,daemi2020chern, HPC, NST} or certain twist knots \cite{NST} (see   \cite{DISST} for very recent results in this direction). In particular, although the instanton approach is well-suited to understanding $\{P(K_n)\}_{n \in \N}$ for $\{K_n\}_{n \in \N}$  a family of distinct torus knots, it is not apparent how to extend this to self-connected sums of a single torus knot, in regards to Conjecture~\ref{conj:1.2}.

In this article, we use recent advances in involutive Heegaard Floer homology \cite{HM, DHSTcobord, HHSZ} to verify Conjecture~\ref{conj:1.2} for all proper rational unknotting number one patterns satisfying a certain non-zero linking number condition; see Theorem~\ref{thm:1.7}. (This class includes all multiply-clasped and twisted Whitehead doubles.) More broadly, for such patterns we establish a general condition on a family of companion knots $\{K_n\}_{n \in \N}$ which guarantees that $\{P(K_n)\}_{n \in \N}$ has infinite rank; see Theorem~\ref{thm:1.9}. Applying this to self-connected sums of a given knot $K$ allows us to verify Conjecture~\ref{conj:1.2} in the cases at hand. In fact, we show that for our examples, $K$ may be chosen to be topologically slice, so that the rank-expanding behavior of Conjecture~\ref{conj:1.2} persists even after restricting $P$ to $\smash{\cC_{TS}}$. 

Even in the well-studied case of Whitehead doubles, our formalism can handle many new families of companion knots. In addition to the linear independence of Whitehead doubles of the form $\{\Wh(nK)\}_{n \in \N}$, we give the first example of an infinite linearly independent family of (positively-clasped) Whitehead doubles whose companion knots all have $\tau(K) \leq 0$. By work of the second author \cite[Theorem 1.7]{Hedden} combined with that of Hom \cite{Hom2017survey} and Sato \cite[Theorem 1.2]{Sato}, if $\tau(K) \leq 0$ then the stable equivalence class of the knot Floer homology of $\Wh(K)$ is trivial. This means that the (non-involutive) knot Floer invariants of $D(K)$ contain no interesting concordance information; hence such knots are difficult to approach directly using knot Floer homology. Note that the companion knots of \cite{HK, PC, HPC, NST} all have positive $\tau$-invariant.

\begin{corollary}\label{cor:B}
There exists a family of knots $\{K_n\}_{n \in \N}$ with each $\tau(K_n) \leq 0$ such that $\{\Wh(K_n)\}_{n \in \N}$ is linearly independent. If desired, the $K_n$ may be taken to be topologically slice.
\end{corollary}
Finally, we provide a re-proof of a conjecture of the second author and Pinz\'on-Caicedo \cite{HPC}, who asked whether there is a knot $K$ such that $D(K)$ and $D(-K)$ are both non-zero in concordance. This was recently answered by Lewark and Zibrowius using Khovanov homology \cite[Corollary 1.13]{LZ}; in Corollary~\ref{cor:6.1} we give a general condition on $K$ which guarantees the linear independence of $D(K)$ and $D(-K)$. In Corollary~\ref{cor:B}, the knots $K_n$ can be taken so that each pair $D(K_n)$ and $D(-K_n)$ are linearly independent.



\subsection{Main theorems}\label{sec:1.1}

We first give a rough overview of the class of patterns considered in this paper. A pattern $P$ has \textit{rational unknotting number one} if there exists a rational tangle $T$ embedded in $P$ such that replacing $T$ with another rational tangle $T'$ gives an unknot in the solid torus. This replacement is said to be \textit{proper} if $T'$ connects the same two pairs of points as $T$. Some examples of rational unknotting number one patterns are given in Figure~\ref{fig:1.1}; see Sections~\ref{sec:2.2} and \ref{sec:2.4} for further discussion and examples.

\begin{figure}[h!]
\center
\includegraphics[scale=1]{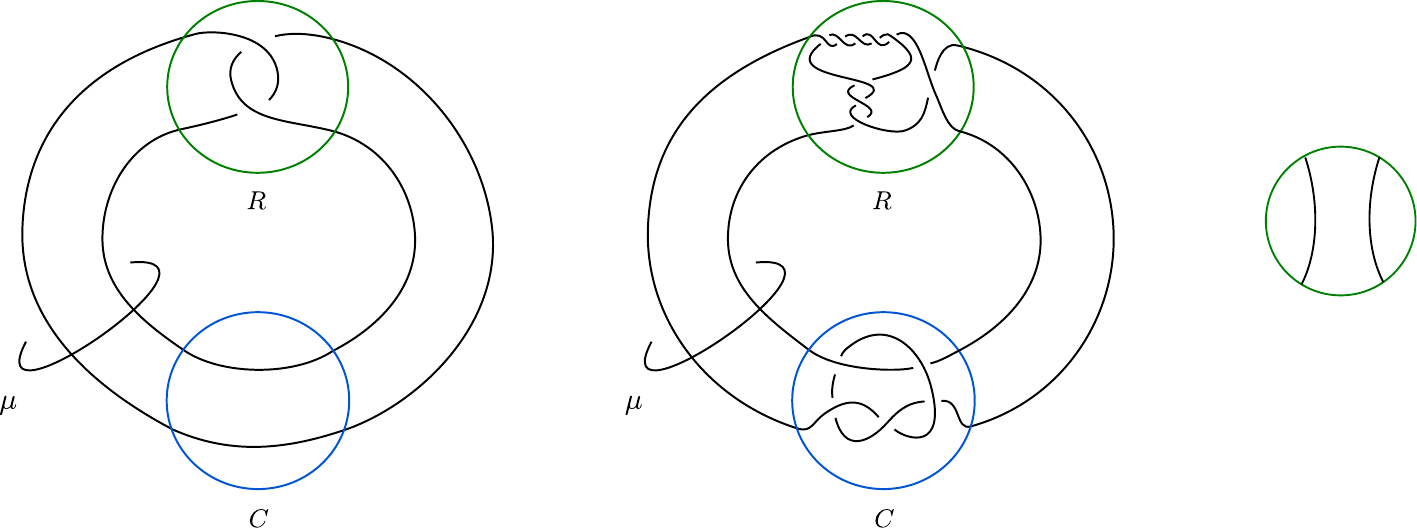}
\caption{A large class of unknotting number one patterns can be formed by gluing a rational tangle $R$ to another tangle $C$ with unknotted (horizontal) closure; see Section~\ref{sec:2.4}. In such cases the unknotting tangle replacement is given by replacing $R$ with a trivial tangle of two vertical strands.}\label{fig:1.1}
\end{figure}

Let $P$ be a rational unknotting number one pattern. By the Montesinos trick, a choice of unknotting tangle replacement identifies the branched double cover $\Sigma_2(P(U))$ with surgery on a strongly invertible knot $J$:
\[
\Sigma_2(P(U)) \cong S^3_{p/q}(J).
\]
In Section~\ref{sec:2.2}, we describe how to explicitly obtain $J$ and the surgery coefficient $p/q$. This data depends on our choice of unknotting tangle replacement; when discussing a rational unknotting number one pattern, we will usually have a fixed tangle replacement in mind, although we suppress writing this explicitly. 

Using this identification, we define an additional invariant of a rational unknotting number one pattern, which we call the \textit{linking number} $\ell$. Let $\mu$ be a meridian of the solid torus for $P$ and $\tmu$ be a lift of $\mu$ to the branched cover $\Sigma_2(P(U))$. We then set $\ell = \lk(J, \tmu)$. This may be computed by using the algorithm for determining $J$ outlined in Section~\ref{sec:2.2}. 

We now state our main theorem:

\begin{theorem}\label{thm:1.7}
Let $P$ be a proper rational unknotting number one pattern with non-zero linking number. Then $P$ is rank-expanding. Moreover, if $K$ is any knot such that $V_0(nK) - V_0(-nK) \rightarrow \infty$ as $n \rightarrow \infty$, then $P$ is rank-expanding along $K$.
\end{theorem}

Again, all multiply-clasped and twisted Whitehead doubles satisfy the hypotheses of Theorem~\ref{thm:1.7}. Since we may freely replace $K$ with $-K$ for the purposes of rank expansion, note that in the latter half of the theorem it also suffices to establish $V_0(nK) - V_0(-nK) \rightarrow \infty$ as $n \rightarrow - \infty$. The condition on $\ell$ is equivalent to the linking number condition of \cite{HPC} and is in fact a property of $P$, independent of the choice of tangle replacement (see Remark~\ref{rem:3.5}).

Although the Floer-theoretic condition $V_0(nK) - V_0(-nK) \rightarrow \infty$ might seem slightly opaque, there are many classes of knots for which this hypothesis is easy to verify. These include the following large families:

\begin{enumerate}
\item $K$ is any L-space knot, such as a torus knot or algebraic knot, or any linear combination of such knots of the same sign/handedness;
\item $K$ is any thin knot with $\tau(K) \neq 0$, such as a (quasi-)alternating knot of non-zero signature;
\item $K$ is any linear combination of genus one knots such that the overall connected sum satisfies $\tau(K) \neq 0$.
\end{enumerate}
These examples are discussed in Section~\ref{sec:5}; note that the above list is certainly not exhaustive. The wide applicability of Theorem~\ref{thm:1.7} may be taken as evidence that Conjecture~\ref{conj:1.2} indeed holds along every rank-one subgroup.

Note that since any Whitehead double has genus one, in Theorem~\ref{thm:1.7} we may take $K$ itself to be a Whitehead double so long as $\tau(K) \neq 0$. This additional condition is quite mild, and can easily be verified using \cite[Theorem 1.4]{Hedden} (cf. \cite{LivingstonNaik}). Setting (for example) $K = \Wh(T_{2,3})$, we immediately obtain:

\begin{corollary}\label{cor:1.8}
Let $P$ be a proper rational unknotting number one pattern with non-zero linking number. Then $P$ is rank-expanding when restricted to the subgroup $\smash{\cC_{TS}}$ of topologically slice knots. Setting $P$ itself to be $D$ (so that the image of $P$ is contained in $\smash{\cC_{TS}}$) gives an example of a rank-expanding operator $P |_{\smash{\cC_{TS}}} \colon \smash{\cC_{TS}} \rightarrow \smash{\cC_{TS}}$.
\end{corollary}

Theorem~\ref{thm:1.7} is a special case of a broader statement regarding the images of general families of companions:

\begin{theorem}\label{thm:1.9}
Let $P$ be a proper rational unknotting number one pattern with non-zero linking number and $p/q > 0$. If $\{K_n\}_{n \in \N}$ is any family of knots such that $V_0(K_n) - V_0(-K_n) \rightarrow \infty$ as $n \rightarrow \infty$, then $\{P(K_{n})\}_{n \in \N}$ has infinite rank.
\end{theorem}
\noindent
Theorem~\ref{thm:1.7} follows immediately from Theorem~\ref{thm:1.9} by setting $K_n = nK$ and (if needed) replacing $P$ by $-P$. (If $P$ has $p/q > 0$, then the mirrored pattern $-P$ has $p/q < 0$.) Note that the families of knots discussed after Theorem~\ref{thm:1.7} all apply to Theorem~\ref{thm:1.9}. Previous results in the vein of Theorem~\ref{thm:1.9} have generally focused on families of companions such as torus knots; the classes discussed in this section are significantly broader.

In certain cases, it is possible to strengthen Theorem~\ref{thm:1.9} by establishing linear independence of the entire image $\{P(K_{n})\}_{n \in \N}$. For this, we restrict to the class of \textit{rational tangle patterns}. We define the $p/q$-rational tangle pattern by taking the closure of a $p/q$-rational tangle, as discussed in Section~\ref{sec:2.4}. This is the simplest case of a rational unknotting number one pattern and corresponds to the case where $J$ is an unknot.

\begin{theorem}\label{thm:1.10}
Let $P$ be a $p/q$-rational tangle pattern with $p/q > 0$.
\begin{enumerate}
\item Suppose $q$ is even. Let $\{K_n\}_{n \in \N}$ be any family of thin knots with $\tau(K_n)$ distinct and greater than $\lfloor(\lfloor p/q \rfloor + 1)/4 \rfloor$. Then $\{P(K_{n})\}_{n \in \N}$ is linearly independent and in fact spans a $\Z^\infty$-summand of $\cC$.
\item Suppose $q$ is odd.  Let $\{K_n\}_{n \in \N}$ be any family of thin knots with $\tau(K_n)$ distinct and less than zero. Then $\{P(K_{n})\}_{n \in \N}$ is linearly independent and in fact spans a $\Z^\infty$-summand of $\cC$.
\end{enumerate}
\end{theorem}
\noindent
A rational tangle pattern always has rational unknotting number one. However, the tangle replacement is proper if and only if $q$ is even. Note that every multiply-clasped and twisted Whitehead double is rational tangle pattern; hence Theorem~\ref{thm:1.10} recovers \cite[Theorem 13]{PC}. Using the methods of this paper, it is also possible to extend Theorem~\ref{thm:1.10} to other (non-thin) classes of companions, including certain families of torus knots or L-space knots.



\subsection{Overview}\label{sec:1.2}

Our results employ the well-established strategy of translating the linear independence of satellites to a question about branched double covers. Recall that taking the branched double cover gives a homomorphism
\[
\Sigma_2 \colon \cC \rightarrow \Theta^3_{\Z_2}.
\]
Thus, to determine whether a given family of knots is linearly independent in $\cC$, it suffices to show that their branched double covers are linearly independent in $\smash{\Theta^3_{\Z_2}}$. Establishing linear independence in the homology cobordism group is an old and well-explored application of Floer homology, and the results of \cite{HK, PC, HPC, NST} have all relied on leveraging the Chern-Simons filtration on instanton Floer theory in this setting. In this paper, we instead use the involutive Heegaard Floer package of Hendricks and Manolescu \cite{HM}. This has already been employed by several authors to study homology cobordism; see for example \cite{HMZ, DaiManolescu, DaiStoffregen, HHL, DHSTcobord, HHSZ}. 

The Heegaard Floer framework is especially suited to this strategy. Indeed, let $P$ be any rational unknotting number one pattern. We show in Section~\ref{sec:2.3} that for any companion knot $K$, the branched double cover $\smash{\Sigma_2(P(K))}$ is homeomorphic to $p/q$-surgery on a certain knot $J_{K, \tmu}$ constructed from $J$ and $K$. To establish the linear independence of $\{P(K_n)\}_{n \in \N}$, it thus suffices to show that the family of $\Z_2$-homology spheres $\{\smash{S^3_{p/q}(J_{K_n, \tmu})}\}_{n \in \N}$ is linearly independent. In joint work with Hendricks, Hom, and Zemke \cite{HHSZ}, the fourth author established a surgery formula for involutive Heegaard Floer homology. Our approach is to use this surgery formula to analyze the involutive Floer homology of $\{\smash{S^3_{p/q}(J_{K_n, \tmu})}\}_{n \in \N}$ and apply existing involutive Floer techniques to show that this family has infinite rank.

\subsection{Comparison with other techniques}

It may be somewhat surprising that involutive Heegaard Floer theory can be used to study the classes of satellites at hand. Indeed, prior to this article, Heegaard Floer invariants had not been successfully employed to establish that {\em any} winding number zero satellite operators have infinite rank. In particular, a host of Heegaard Floer theoretic invariants had failed a simple litmus test in this direction; namely, (re)proving the independence of infinite families of (untwisted) Whitehead doubles, first exhibited in \cite{HK}.

It is also worth noting an interesting conceptual distinction between the instanton and Heegaard Floer homologies. An important feature in the realm of instanton Floer homology is its filtration by the Chern-Simons functional, which provides refined topological invariants that are crucial for the arguments of \cite{HK, PC,daemi2020chern, HPC, NST}. While instanton and Heegaard Floer homology share many formal properties, no such filtration is present on the Heegaard Floer side.  Indeed, the analagous filtration on the Heegaard Floer side comes from the symplectic action functional used in the definition of Lagrangian Floer homology for the symmetric product of a Heegaard diagram. To date, however, no topological significance of this information for 3-manifolds and cobordisms between them  has been discovered.   Even if the action functional  could be used in a similar manner,  it seems unlikely that the Heegaard Floer package, being isomorphic to an abelian gauge theoretic  Floer theory (Seiberg-Witten monopole Floer homology) could recover the topological information about  non-abelian fundamental group representations contained in the Chern-Simons filtration. It is thus curious that the usage of involutive Heegaard Floer homology in our situation suffices to recover (and in some cases extend) previously known results established using instanton Floer theory.  

\subsection*{Organization}
In Section~\ref{sec:2}, we introduce the notion of a rational unknotting number one pattern and review the basic setup of involutive Heegaard Floer homology and local equivalence. We then prove Theorems~\ref{thm:1.7} and \ref{thm:1.9} in Section~\ref{sec:3} and Theorem~\ref{thm:1.10} in Section~\ref{sec:4}. In Section~\ref{sec:5}, we give some examples of Theorems~\ref{thm:1.7} and \ref{thm:1.9}. Finally, in Section~\ref{sec:6} we prove Corollaries~\ref{cor:A} and \ref{cor:B} and discuss further applications to Whitehead doubles.

\subsection*{Acknowledgements} ID was supported by NSF grant DMS-1902746.  MH was supported by NSF grant DMS-2104664. MS was supported by NSF grant DMS-1952755. MS was supported by NSF grant DMS-1952755. This material is based upon work supported by the National Science Foundation under grants  DMS-1929284 while MH was in residence at the Institute for Computational and Experimental Research in Mathematics in Providence, RI, during the Braids program in Spring 2022, and under grant DMS-1928930 while ID, MH, and AM were in residence at the Simons Laufer Mathematical Sciences Institute (previously MSRI) in Berkeley, CA, during the Fall 2022 semester. AM was supported by the postdotoral fellowship from Max-Planck-Institut f\"ur Mathematik and SLMath (MSRI) during the course of this work. MS thanks Kristen Hendricks, Jenifer Hom, Ian Zemke for their collaboration with him during work on the involutive surgery formula. The authors also Chuck Livingston for helpful conversations. 

\section{Background}\label{sec:2}

In this section, we define the class of rational unknotting number one patterns and give a brief overview of the setup of involutive Heegaard Floer homology.

\subsection{Rational tangles}\label{sec:2.1}
We first review the notion of a rational tangle. Let $B^3$ be a $3$-ball with four marked points on its boundary. A \textit{Conway tangle} (or sometimes just \textit{tangle}) is a proper embedding of two disjoint arcs $T \subseteq B^3$ whose boundaries are precisely the four marked points. Two tangles are \textit{isotopic} if there is an isotopy fixing the boundary which takes one to the other.

\begin{definition}\label{def:2.1}
A tangle is \textit{rational} if it consists of a pair of boundary-parallel arcs.
\end{definition}
 
The set of rational tangles in a fixed $3$-ball $B^3$ may be placed in (non-canonical) bijection with $\Q \cup \{\infty\}$, as follows. Fix a projection of $B^3$ and let $\smash{T_{1/0} = T_\infty}$ and $\smash{T_{0/1} = T_0}$ be the tangles displayed in Figure~\ref{fig:2.1}. Given any $p/q \in \Q \cup \{\infty\}$, consider the continued fraction
\[
p/q = [x_1, x_2, \cdots , x_n]= x_1+\cfrac{1}{x_2+\cfrac{1}{x_3+ \cdots + \cfrac{1}{x_{n}}}}
\]
with each $x_i \in \Z$. Let $h$ and $v$ be the horizontal and vertical half-twist operations displayed on the right in Figure~\ref{fig:2.1}. We then define the $p/q$-rational tangle $T_{p/q}$ to be
\[
    T_{p/q} = 
\begin{cases}
    h^{x_1} v^{x_2} \cdots h^{x_{n-1}} v^{x_n} T_{\infty}& \text{for \textit{n} even}\\
    h^{x_1} v^{x_2} \cdots v^{x_{n-1}} h^{x_n} T_{0}& \text{for \textit{n} odd}.
\end{cases}
\]

\begin{figure}[h!]
\center
\includegraphics[scale=0.8]{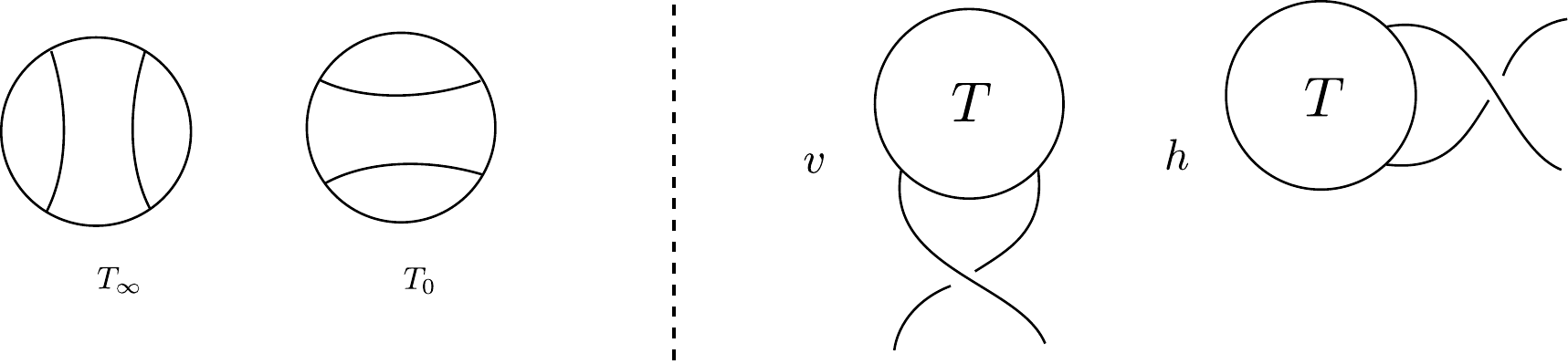}
\caption{Left: the tangles $\smash{T_{1/0} = T_\infty}$ and $\smash{T_{0/1} = T_0}$. Right: adding half-twists to a tangle $T$ via the operations $v$ and $h$.}\label{fig:2.1}
\end{figure}

Conway \cite{Conway} showed that up to isotopy, $\smash{T_{p/q}}$ is independent of the choice of the continued fraction decomposition of $p/q$ and that every rational tangle (on the same marked $3$-ball) arises from the above construction. (Here, our sign convention is opposite to that in \cite{Gordon}.) See Figure~\ref{fig:2.2} for some examples of rational tangles.

\begin{figure}[h!]
\center
\includegraphics[scale=0.9]{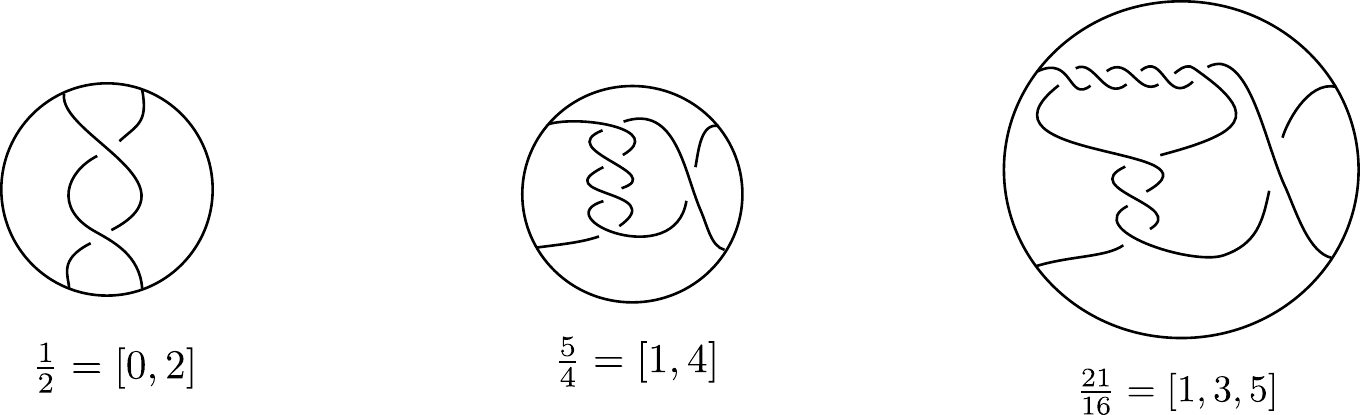}
\caption{Examples of rational tangles.}\label{fig:2.2}
\end{figure}

We stress that identifying a rational tangle with an element of $\Q \cup \{\infty\}$ is relative to a particular projection; or, equivalently, a choice for $T_\infty$ and $T_0$. (In the next subsection, we will see why being precise with this identification is so important.) Indeed, given an abstract $3$-ball with four marked points, there is no canonical choice for $T_\infty$ or $T_0$ without fixing a preferred projection. Instead, we declare $T_\infty$ and $T_0$ to be a pair of rational tangles which connect different pairs of marked points on $\partial B^3$ and are simultaneously boundary-parallel. Pushing $\smash{T_\infty}$ and $T_0$ to the boundary of $B^3$ then divides $\partial B^3$ into two hemispheres, from which it easily follows that up to homeomorphism (not fixing $\partial B^3$) we may draw $\smash{T_\infty}$ and $T_0$ as in Figure~\ref{fig:2.1}. When we refer to a \textit{$p/q$-rational tangle} without further elaboration, we will usually have in mind the standard projection in the sense of Figure~\ref{fig:2.1}.

In general, if we have fixed a projection of $B^3$ in which $\smash{T_\infty}$ and $T_0$ are not standard (in the sense of Figure~\ref{fig:2.1}), then in order to identify a tangle $T$ with an element of $\Q \cup \{\infty\}$, we must find the homeomorphism $F: B^3 \rightarrow B^3$ which moves $\smash{T_\infty}$ and $T_0$ into standard position with respect to the projection. We then apply the previous discussion to the projection of $F(T)$. 

\subsection{Rational unknotting number one patterns}\label{sec:2.2}
We now define the class of patterns considered in this paper.

\begin{definition}\label{def:2.2}
Let $P \subseteq S^1 \times D^2$ be a pattern. We say that $P$ has \textit{rational unknotting number one} if there exists a rational tangle $T$ in $P$ such that replacing $T$ with another rational tangle $T'$ gives a knot which is unknotted in the solid torus. We say that $P$ has \textit{proper rational unknotting number one} if $T'$ can be taken to be a proper tangle replacement; that is, connecting the same two pairs of marked points as $T$.
\end{definition}
See Figure~\ref{fig:2.3} for an example of a rational unknotting one pattern. We will write $P'$ to refer to the result of replacing $T$ with $T'$; this is of course isotopic to the unknot. When we discuss a rational unknotting number one pattern, we will usually implicitly have a particular unknotting tangle replacement $T'$ in mind, although a single pattern may admit several different unknotting replacements.

For us, the important feature of a rational unknotting number one pattern is that its branched double cover is surgery on a strongly invertible knot. Recall that a knot $J$ is called \textit{strongly invertible} if there exists an orientation-preserving involution $\tau$ of $S^3$ which fixes $J$ setwise and has two fixed points on $J$. By \cite{Waldhausen}, it follows that $\tau$ is conjugate to $180^{\circ}$ rotation about an unknotted axis. We claim that if $P$ has rational unknotting number one, then
\[
\Sigma_2(P(U)) \cong S^3_{p/q}(J)
\]
for some strongly invertible knot $J$ and surgery coefficient $p/q$. Moreover, this homeomorphism identifies the branched covering action on $\Sigma_2(P(U))$ with the involution on $\smash{S^3_{p/q}(J)}$ induced by the strong inversion on $J$. Our claim is immediate from the Montesinos trick: since $P'$ is an unknot, the branched double cover over $P'$ is $S^3$. The $3$-ball $B^3$ containing $T'$ lifts to a solid torus in $S^3$, and replacing $T'$ with $T$ corresponds to doing surgery on the core of this solid torus. 

However, explicitly producing $J$ and the surgery coefficient $p/q$ is slightly involved. An example of this procedure is given in Figure~\ref{fig:2.3}. Here, we have drawn the tangles $T$ and $T'$ in black and red, respectively, while $B^3$ is drawn in green. It is straightforward to check that replacing $T$ with $T'$ gives an unknot in the solid torus. The meridian of $S^1 \times D^2$ is labeled $\mu$.

To draw $J$, let $\gamma$ be a reference arc in $B^3$ which has one endpoint on each component of $T'$. In general, there are many such arcs (each looping around the components of $T'$ multiple times); we select one by requiring $\gamma$ not to intersect the disks obtained as traces of the isotopy pushing $T'$ to $\partial B^3$. In the case that $T'$ has trivial projection, $\gamma$ is the obvious arc running from one component to the other, as displayed in panel (2) of Figure~\ref{fig:2.3}. Let $F_t$ be an isotopy of the solid torus moving $P'$ into a local unknot in $S^1 \times D^2$. Apply $F_1$ to $\gamma$ and $B^3$, as shown in (3) of Figure~\ref{fig:2.3}. It is then straightforward to draw the lift of $F_1(\gamma)$ to the branched double cover over the unknot $F_1(P')$. This gives the desired strongly invertible knot $J$, displayed in panel (4). 

\begin{figure}[h!]
\center
\includegraphics[scale=0.8]{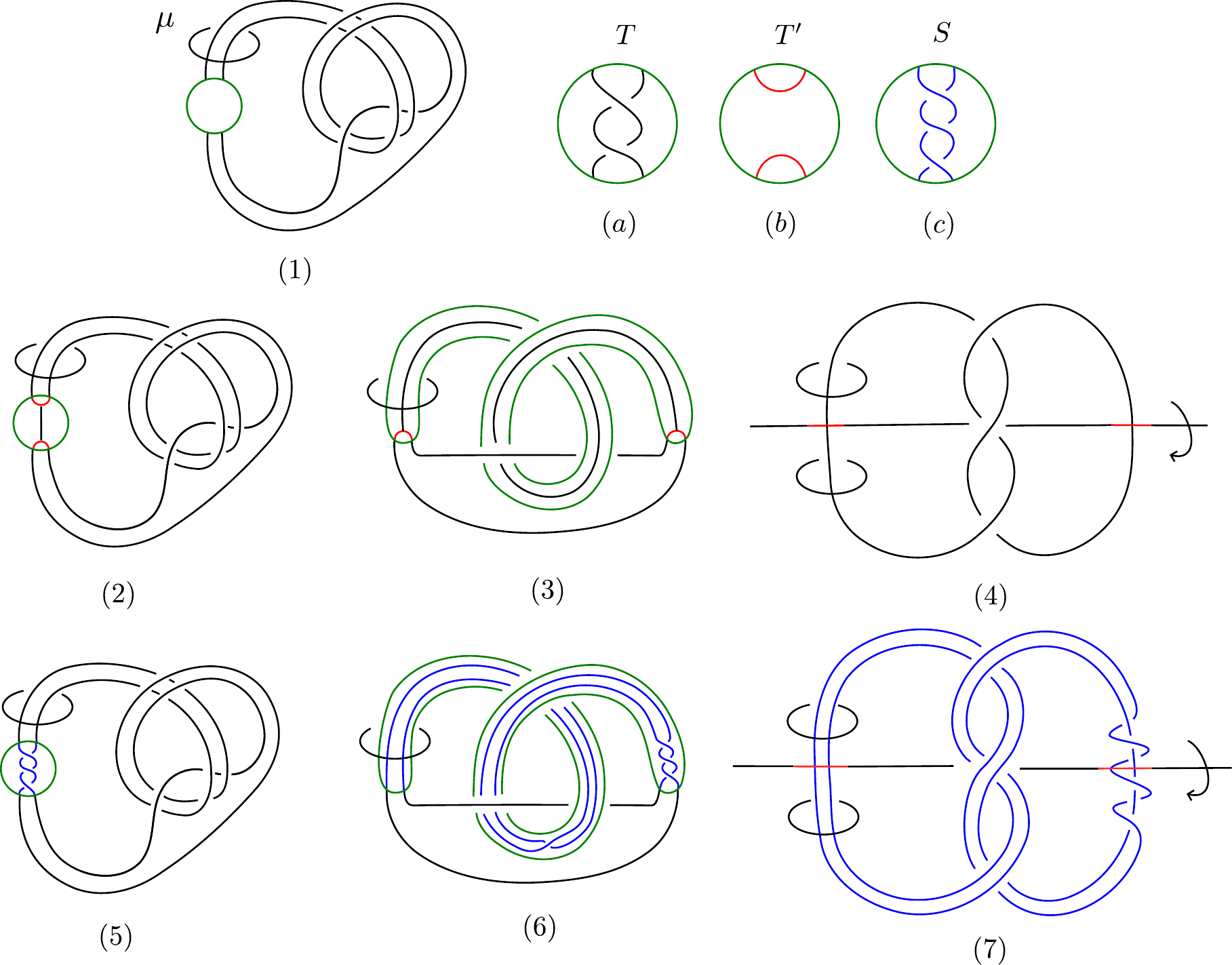}
\caption{Top row:  ($1$) the pattern $P$; ($a$), ($b$), and ($c$) are the tangles $T$, $T'$, and $S$. Middle row: ($2$) the pattern $P'$ and arc $\gamma$; ($3$) applying the isotopy $F_1$ to $P'$ and $\gamma$; ($4$) the strongly invertible knot $J$. Bottom row: ($5$) the tangle $S$; ($6$) applying the isotopy $F_1$ to $S$; ($7$) the $\tau$-invariant Seifert framings of $J$.}\label{fig:2.3}
\end{figure}

Determining the surgery coefficient $p/q$ is slightly more involved. In order to do this, we must find the unique rational tangle $S$ in $B^3$ which lifts to a pair of $\tau$-equivariant Seifert framings of $J$; this is colored blue in Figure~\ref{fig:2.3}. Determining $S$ can be done by running $F_t$ backwards: in Figure~\ref{fig:2.3}, panel (7) shows the two $\tau$-invariant Seifert framings for $J$. The quotient of these by $\tau$ is displayed in (6), while in (5) we have reversed the isotopy $F_t$ to draw $S$ in the original $3$-ball $B^3$. By the Montesinos trick, the surgery coefficient $p/q$ is then precisely the rational number identified with the original tangle $T$ relative to the choice of reference tangles $\smash{T_\infty} = T'$ and $T_0 = S$. 

\begin{definition}\label{def:2.3}
Let $P$ be a rational unknotting number one tangle with a fixed choice of unknotting tangle replacement $T'$. As discussed above, this gives an identification
\[
\Sigma_2(P(U)) \cong S^3_{p/q}(J).
\]
We refer to $p/q$ as the \textit{coefficient} of $P$. Note that because $P$ is a knot, its branched double cover is a $\Z_2$-homology sphere; hence $p$ is necessarily odd. We say that $P$ is \textit{even} or \textit{odd} according to the parity of $q$ and \textit{positive} or \textit{negative} according to the sign of $p/q$.
\end{definition}

In our setting, it turns out that even and odd rational unknotting number one patterns behave rather differently. Fortunately, even though determining the exact coefficient $p/q$ of $P$ is rather difficult, the parity of $q$ can be easily read off from the tangle replacement:
\begin{lemma}{\cite[Corollary 2]{MZ}}\label{lem:2.4}
Let $P$ be a rational unknotting number one pattern with a fixed choice of unknotting tangle replacement $T'$. Then $q$ is even if and only if the tangle replacement is proper.
\end{lemma}

\begin{proof}
As in the discussion of Section~\ref{sec:2.1}, let $F$ be a homeomorphism of $B^3$ taking $T'$ and $S$ to the standard $\infty$- and $0$-tangles $T_\infty$ and $T_0$, respectively. Note that the tangle replacement $T$ to $T^{\prime}$ is proper if and only if the tangle replacement $F(T)$ (which is by definition the $p/q$-tangle with respect to the standard $\infty$- and $0$-tangles) to $T_\infty$ is proper. It was shown in \cite[Lemma 11]{MZ} that the distance between $T_{p/q}$ and $T_\infty$ is even if and only if the replacement $T_{p/q}$ to $T_\infty$ is proper. The result follows by observing that the distance between $T_{p/q}$ and the $T_\infty$ is precisely $q$.
\end{proof} 

Note that if $P$ is a negative pattern, then its mirror is positive. We thus lose no generality in considering the class of positive rational unknotting number one patterns. 

There is one more important piece of data associated to a rational unknotting number one pattern. This is the following:
\begin{definition}\label{def:2.5}
Let $P$ be a rational unknotting number one tangle with a fixed choice of unknotting tangle replacement $T'$. We define the \textit{linking number} of $P$ by
\[
\ell = \lk(J, \tmu).
\]
As $\tmu$ and $\tau \tmu$ are equally preferenced, $\ell$ is defined only up to sign. (Note that $\ell$ is well-defined since $J$ is independent of the choice of isotopy $F_t$, up to equivariant homeomorphism.)
\end{definition}


\subsection{Branched covers of satellites}\label{sec:2.3}

We now extend the discussion of the previous subsection to the branched double cover of $P(K)$ for an arbitrary companion knot $K$. First observe that in the algorithm of Section~\ref{sec:2.2}, the meridian $\mu$ of our pattern lifts to a symmetric unlink $\tmu \cup \tau \tmu$ in $S^3$ disjoint from $J$. Here, the fact that $\tmu \cup \tau \tmu$ is a two-component unlink follows from the condition that $P'$ is unknotted in $S^1 \times D^2$. Note that the data of $P$ comes with an orientation of $\mu$; we give $\tmu \cup \tau \tmu$ the lifted orientation. 
Let $K$ be an oriented knot in $S^3$. Recall that $P(K)$ can be constructed by taking the image of $P$ inside the gluing
\[
S^3 \cong (S^3 - N(\mu)) \cup_{\partial N(\mu)} (S^3 - N(K))
\]
formed by a boundary identification which maps a meridian of $\mu$ to a Seifert framing of $K$ and a longitude of $\mu$ to a meridian of $K$ (respecting the orientations). It follows from the discussion of the previous subsection that 
\[
\Sigma_2(P(K)) \cong (S^3_{p/q}(J) - N(\tmu) - N(\tau \tmu)) \cup_{\partial N(\tmu)} (S^3 - N(K)) \cup_{\partial N(\tau \tmu)} (S^3 - N(K)).
\]
Note that this manifold has an obvious involution. On $\smash{S^3_{p/q}(J) - N(\tmu) - N(\tau \tmu)}$, this involution is induced by the strong inversion on $J$, while elsewhere we simply exchange the two copies of $S^3 - N(K)$.

\begin{definition}\label{def:2.6}
Let $J$ be a strongly invertible knot and $\tmu \cup \tau \tmu$ be an oriented, symmetric unlink disjoint from $J$. We assume that $\tmu \cup \tau \tmu$ is given a $\tau$-invariant orientation. For any oriented knot $K$, define the \textit{double infection} $\smash{J_{K, \tmu}}$ by infecting $J$ twice: once using $K$ along $\tmu$ and once using $K$ along $\tau \tmu$. Note that $\smash{J_{K, \tmu}}$ is a strongly invertible knot.
\end{definition}

\begin{figure}[h!]
\center
\includegraphics[scale=0.9]{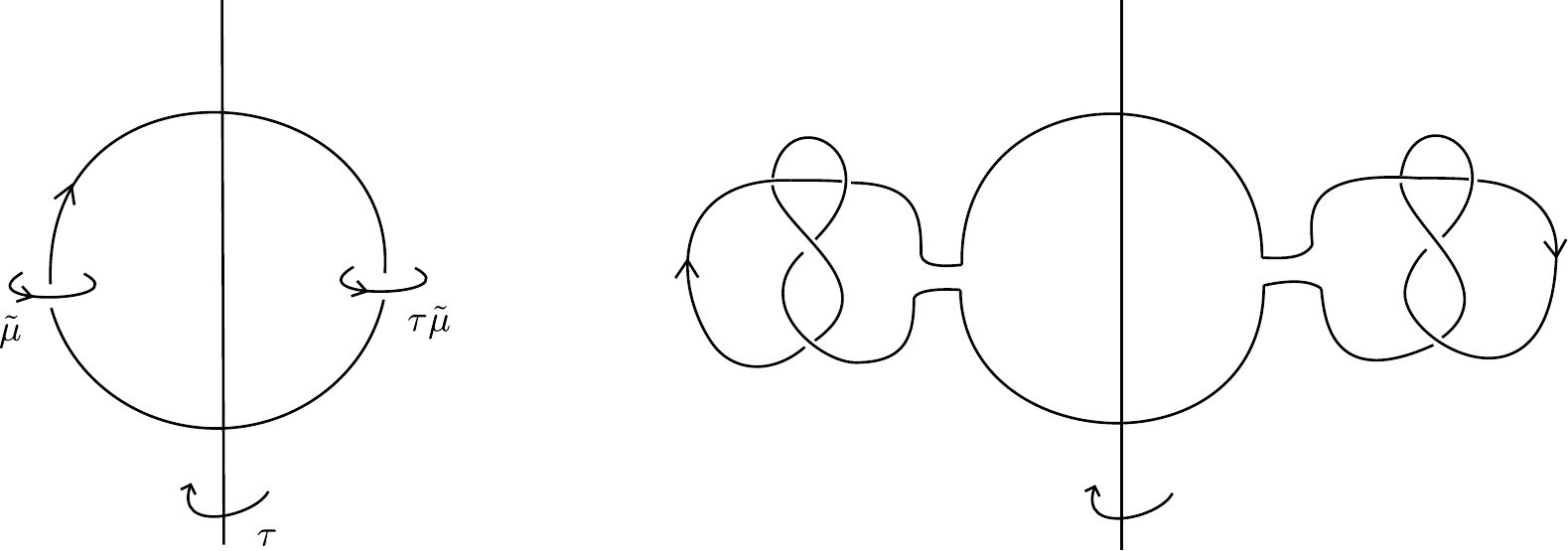}
\caption{In the trivial case where $J$ is the unknot and $\tmu$ is a standard meridian, we have $\smash{J_{K, \tmu}} = K \# K^r$. Note the reversal of orientation in the second factor; this is because $\tmu$ and $\tau \tmu$ are oriented such that $\lk(K, \tmu)  = -\lk(K, \tau \tmu)$.}\label{fig:2.4}
\end{figure}

We thus have:
\[
\Sigma_2(P(K)) \cong S^3_{p/q}(J_{K, \tmu}).
\]
Note that the surgery coefficient $p/q$ is independent of $K$ (and is the same as that of Definition~\ref{def:2.3}). Moreover, once again this homeomorphism identifies the branching action on $\Sigma_2(P(K))$ with the involution on $\smash{S^3_{p/q}(J_{K, \tmu})}$ induced by the strong inversion on $J_{K, \tmu}$. 

To spell out the relevance of this construction, recall that our general strategy is to study the family $\{P(K_n)\}_{n \in \N}$ via their branched double covers $\{\Sigma_2(P(K_n))\}_{n \in \N}$. If $P$ is a rational unknotting number one pattern, then this is the same as studying $p/q$-surgeries on the family of knots
\[
J_n = J_{K_n, \tmu}.
\]
We will use the fact that the $J_n$ are all (double) infections of the same knot in order to derive certain structural results regarding the Floer homologies of these surgeries. This will allow us to establish the desired linear independence.

\subsection{Examples}\label{sec:2.4}
We now give some examples of rational unknotting number one patterns. The simplest of these are \textit{rational tangle patterns}. A rational tangle pattern is obtained by taking the horizontal closure of a $p/q$-rational tangle in the standard projection, as in Figure~\ref{fig:2.5}. Clearly, each such pattern has rational unknotting number one, with the replacement tangle $T'$ being the standard $\infty$-tangle. (The resulting strongly invertible knot $J$ is the unknot.) A rational tangle pattern has linking number $\pm 1$ and surgery coefficient precisely $p/q$. Note that all multiply-clasped, multiply-twisted Whitehead doubling patterns fall into this class.

\begin{figure}[h!]
\center
\includegraphics[scale=0.5]{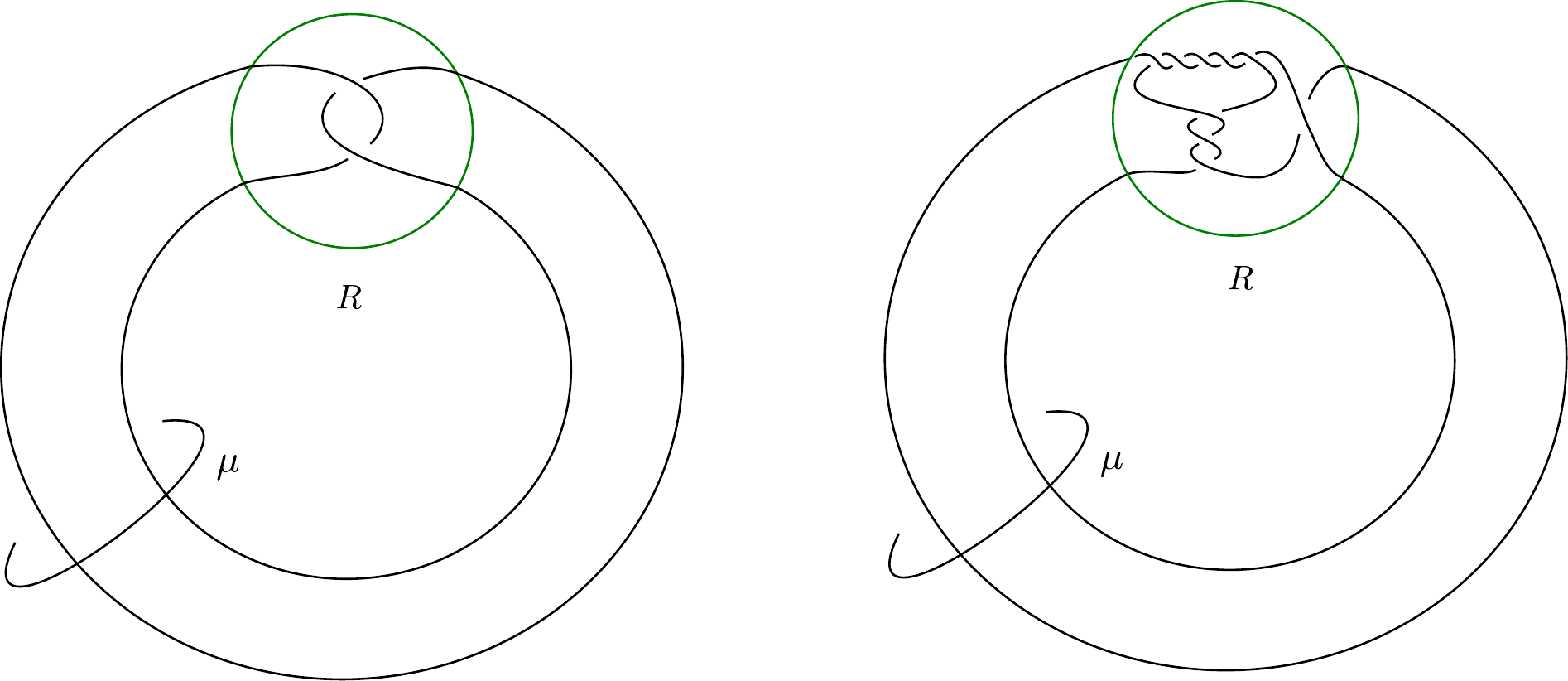}
\caption{The rational tangle patterns corresponding to $1/2$ and $21/16$.}\label{fig:2.5}
\end{figure}

One particularly simple way of constructing a rational unknotting number one pattern is to start from a Conway tangle $C$ with unknotted (horizontal) closure and glue it to any rational tangle $R$, as in Figure~\ref{fig:2.6}. We embed this in $S^1 \times D^2$ by choosing the indicated meridian $\mu$. The resulting pattern tautologically has rational unknotting number one by replacing $R$ with the usual $\infty$-tangle, and some thought shows that $\ell = \pm 1$. Note that rational tangle patterns are a special case of this construction, where $C$ is chosen to be the trivial tangle of two horizontal strands.

\begin{figure}[h!]
\center
\includegraphics[scale=0.5]{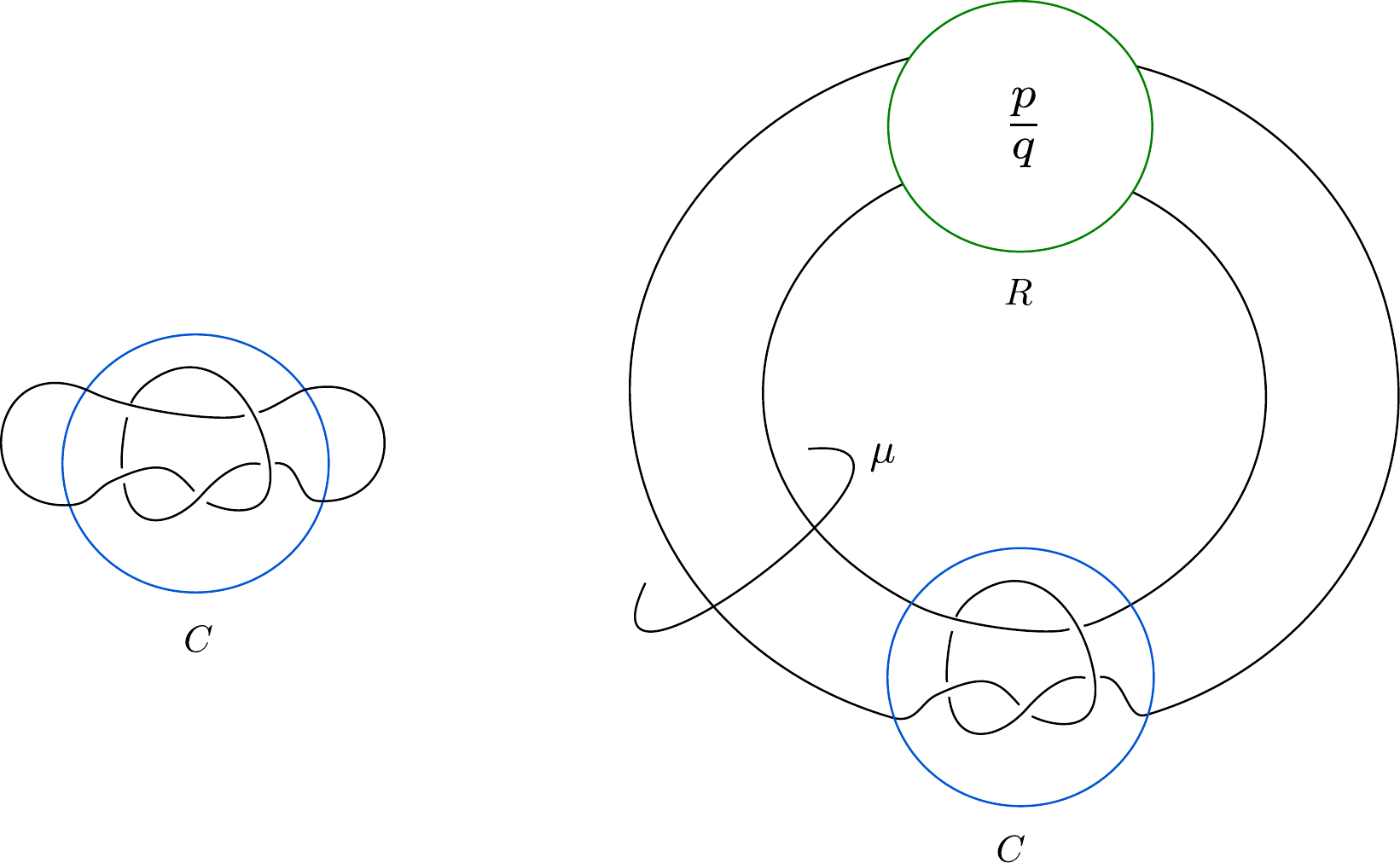}
\caption{Left: A tangle $C$ with an unknotted closure. Right: A rational unknotting number one pattern defined using $C$.}\label{fig:2.6}
\end{figure}

In general, it is often easier to construct a rational unknotting number one pattern by working backwards from the associated knot $J$. As in Definition~\ref{def:2.6}, let $J$ be a strongly invertible knot equipped with a symmetric unlink $\tmu \cup \tau \tmu$ disjoint from $J$. We furthermore require that $\tmu \cup \tau \tmu$ is in fact unlinked in the complement of the axis of symmetry. The algorithm of (for example) \cite[Section 1.1.12]{SavelievInvariants} allows us to express any surgery $\smash{S^3_{p/q}(J)}$ as the branched double cover over a knot. We turn this knot into a pattern by taking the image of $\tmu \cup \tau \tmu$ under the quotient map. Examples of this procedure are shown along the top rows of Figures~\ref{fig:2.7} and \ref{fig:2.8}. 

\begin{figure}[h!]
\center
\includegraphics[scale=.8]{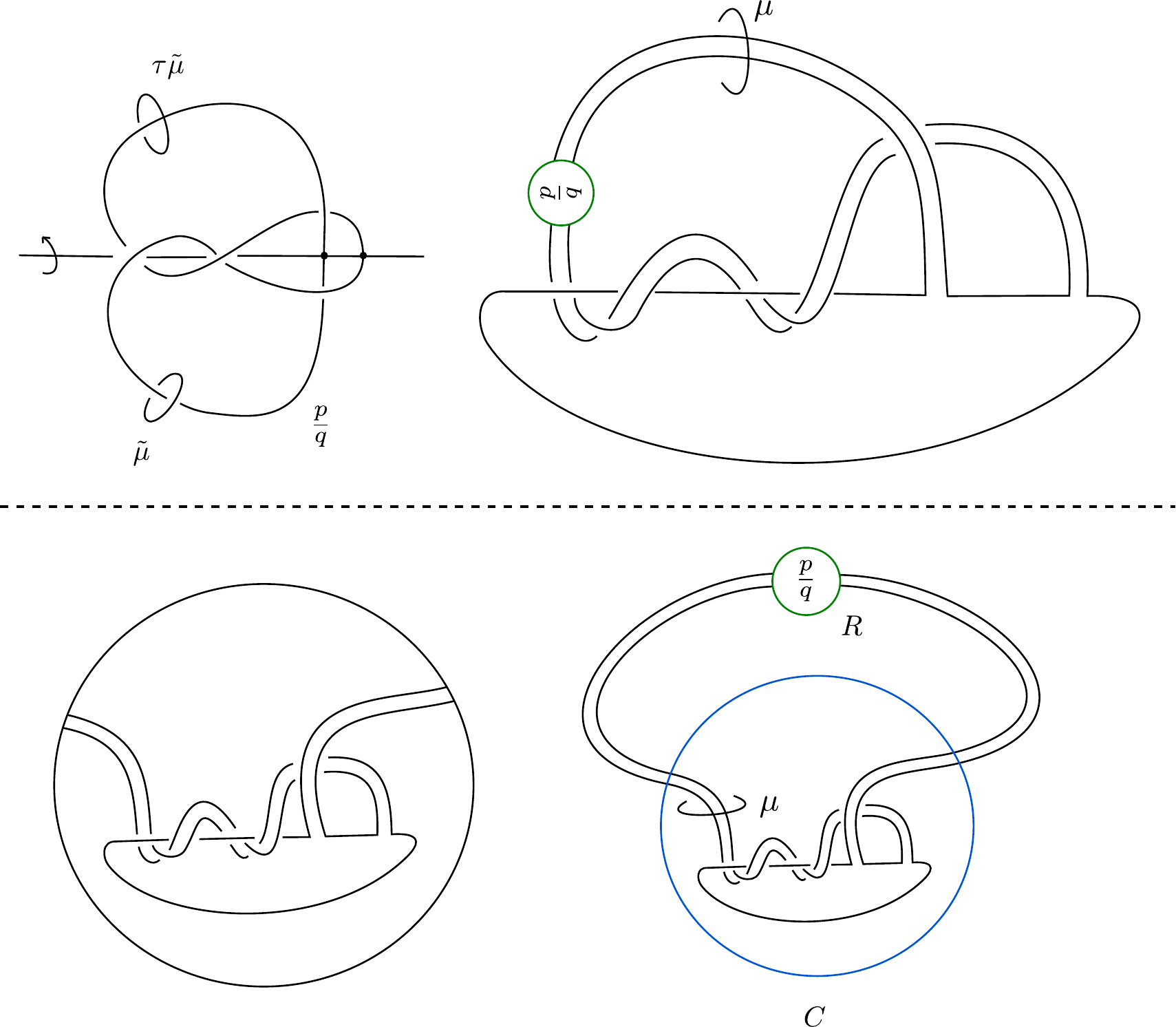}
\caption{A pattern whose double branched cover is surgery on the figure-eight knot, with $\tmu$ a standard meridian. Here, $\ell = 1$.}\label{fig:2.7}
\end{figure}

As shown along the bottom rows of Figures~\ref{fig:2.7} and \ref{fig:2.8}, the result may be viewed as a Conway tangle $C$ glued to a $p/q$-rational tangle $R$ in the standard projection. If $\tmu$ is chosen to be a meridian of $J$ (as in Figure~\ref{fig:2.7}), then the meridian $\mu$ of the resulting pattern will be obviously isotopic to a curve between $C$ and $R$, as in Figure~\ref{fig:2.6} but in general $\mu$ may be more complicated (as in Figure~\ref{fig:2.8}). (Different choices for $\tmu$ give patterns which are the same as knots, but which may have different embeddings in the solid torus.) It is clear from the algorithm of \cite[Section 1.1.12]{SavelievInvariants} that replacing the $p/q$-rational tangle with an $\infty$-tangle gives an unknot in the complement of $\mu$. The same algorithm likewise shows that the tangle $S$ coming from the Seifert framing of $J$ is the usual $0$-tangle in the standard projection. By construction, any such pattern has non-zero linking number as long as $\tilde{\mu}$ was chosen so that $lk(J,\tilde{\mu}) \neq 0$.

\begin{figure}[h!]
\center
\includegraphics[scale=.8]{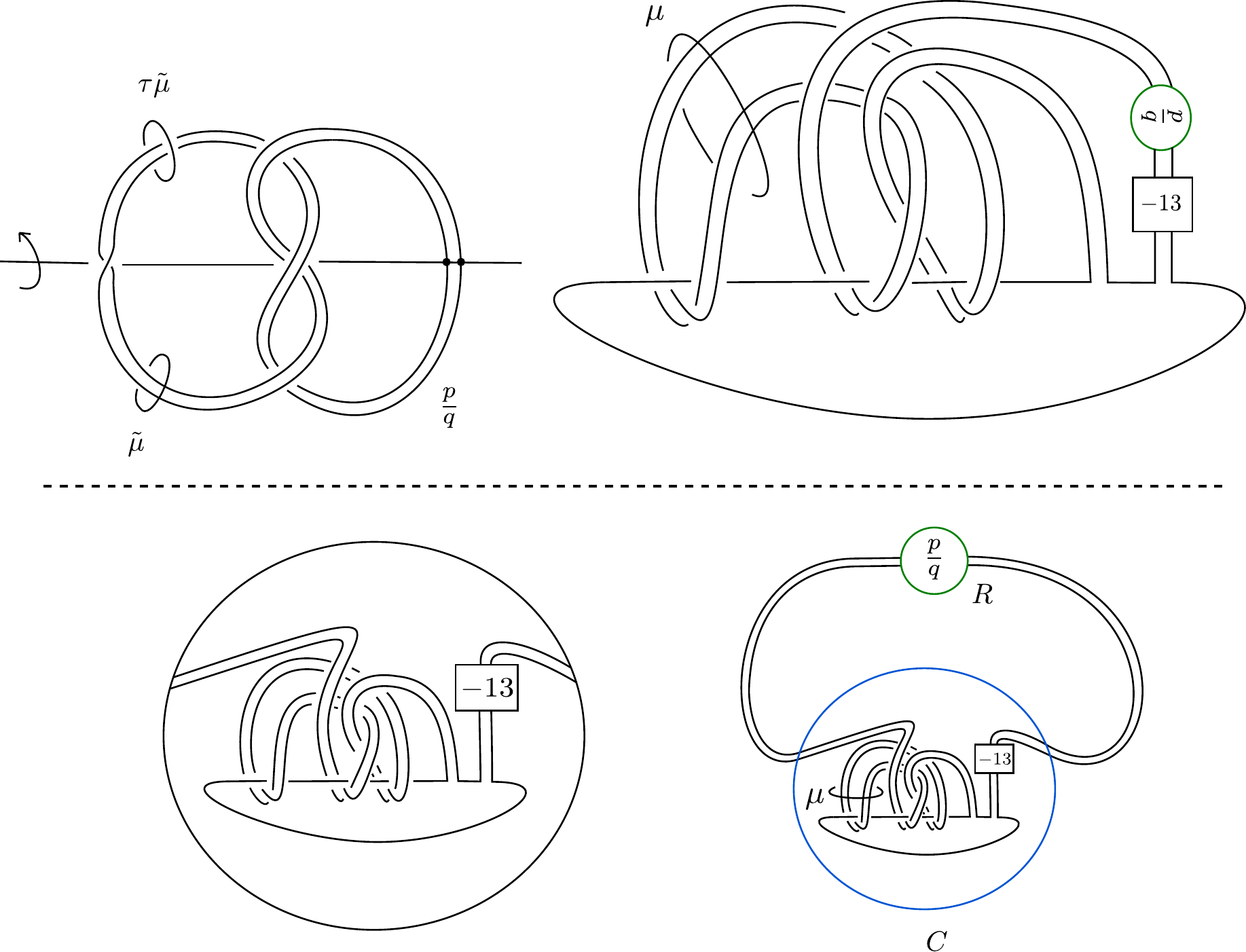}
\caption{A pattern whose double branched cover is surgery on the $(2,1)$-cable of the trefoil, with $\tmu$ as indicated. Here, $\ell = 2$.}\label{fig:2.8}
\end{figure}

\subsection{Local equivalence}\label{sec:2.5}
In this section, we give a brief overview of the involutive Heegaard Floer package and the machinery of local equivalence. Let $Y$ be a 3-manifold equipped with a self-conjugate $\spinc$-structures $\s=\bar{\s}$. In \cite{HM}, Hendricks and Manolescu defined a homotopy involution on the Heegaard Floer chain complex coming from the conjugation symmetry present in a Heegaard diagram:
\[
\iota: \CFm(Y,\s) \rightarrow \CFm(Y,\s).
\]
This additional data provides an enhancement of the usual Heegaard Floer invariant of Ozsv\'ath and Szab\'o \cite{OS3manifolds1, OS3manifolds2}. More precisely, associated to $(Y, \s)$, we may consider the pair $(\CFm(Y, \s), \iota)$; up to an appropriate notion of homotopy equivalence, this is a diffeomorphism invariant of $Y$. 

Although involutive Heegaard Floer homology is defined for all 3-manifolds, we will mainly be concerned with rational homology spheres. In this case, we formalize the resulting algebraic structure in the following definition:

\begin{definition}{\cite[Definition 8.1]{HMZ}}\label{def:2.8}
An {\em $\iota$-complex} is a pair $(C, \iota)$, where:
\begin{enumerate}
\item $C$ is a (free, finitely-generated) chain complex over $\ff[U]$ with
\[
U^{-1}H_*(C) \cong \ff[U, U^{-1}].
\]
Here, $\ff = \Z/2\Z$. We require $C$ to be graded by a coset of $\Z$ in $\Q$ and $U$ to be of degree $-2$.
\item $\iota : C \to C$ is a $\ff[U]$-equivariant, grading-preserving homotopy involution; that is, 
\[
\iota^2 \simeq \id
\]
via a $U$-equivariant chain homotopy.
\end{enumerate}
\end{definition}

In \cite{HMZ}, it is shown that if $Y$ is a rational homology sphere equipped with a self-conjugate $\spinc$-structures $\s$, then the pair $(\CFm(Y, \s), \iota)$ is an $\iota$-complex.

In order to study homology cobordism, we introduce the following equivalence relation:

\begin{definition}{\cite[Definition 8.3]{HMZ}}\label{def:2.7}
Two $\inv$-complexes $(C, \iota)$ and $(C', \iota')$ are called {\em locally equivalent} if there exist $\F[U]$-equivariant, grading-preserving chain maps
\[
f \colon C \to C' \quad \text{and} \quad g \colon C' \to C
\]
such that 
\[
f \circ \iota \simeq \iota' \circ f \quad \text{and} \quad g \circ \iota' \simeq \iota \circ g
\]
and $f$ and $g$ induce isomorphisms on homology after localizing with respect to $U$. We call a map $f$ as above a \textit{local map} from $(C, \iota)$ to $(C', \iota')$, and similarly we refer to $g$ as a local map in the other direction.
\end{definition}

To see the relevance of Definition~\ref{def:2.7} to homology cobordism, let $(Y_1, \s_1)$ and $(Y_2, \s_2)$ be two rational homology spheres equipped with self-conjugate $\spinc$-structures and let $W$ be a rational homology cobordism from $Y_1$ and $Y_2$. Suppose $W$ admits a self-conjugate $\spinc$-structure $\s$ restricting to $\s_i$ on each $Y_i$. Then the Heegaard Floer cobordism map $F_{W, \s}$ (together with its reverse) constitutes a local equivalence between the $\iota$-complexes associated to $(Y_1, \s_1)$ and $(Y_2, \s_2)$.

In the setting of $\Z$- or $\Z_2$-homology cobordism, we always have a unique self-conjugate $\spinc$-structure on $Y$ or $W$. In this case, we may unambiguously associate to $Y$ its \textit{local equivalence class}; this is a homology cobordism invariant. Denote
\[
\Inv = \{\text{all $\iota$-complexes}\}\ /\ \text{local equivalence}.
\]
We then obtain a map $\smash{h \colon \Theta^{3}_{\mathbb{Z}_{2}} \rightarrow \Inv}$ given by sending
\[
[Y] \mapsto h([Y]) = [(\CFm(Y, \s)[-2], \iota)].
\]
Here, $\s$ is the unique self-conjugate $\spinc$-structure on $Y$ and the $[-2]$ is a formal (unimportant) grading shift. In \cite[Section 8.3]{HMZ} it is shown that $\Inv$ is an abelian group with the operation of tensor product. The identity element is given by the trivial complex $\ff[U]$ and inverses are given by dualizing; see \cite[Section 8.3]{HMZ} for details.\footnote{Strictly speaking, our notation $\Inv$ is not quite the group $\Inv$ of \cite[Section 8.3]{HMZ}. The difference is that here, we allow our $\iota$-complex to have gradings valued in $\Q$, rather than $\Z$.} Moreover, it is shown that $h$ is a well-defined homomorphism.

It is thus possible to show that a given family of $\Z_2$-homology spheres is linearly independent by computing their local equivalence classes and establishing their linear independence in $\Inv$. However, this requires an analysis of the algebraic structure of $\Inv$. Although in general $\Inv$ is very complicated, techniques for carrying out this strategy have been developed in (for example) \cite{DaiManolescu, DaiStoffregen, HHL, DHSTcobord}. The results of the current article will depend on several such calculations, which we outline in the next subsection.

\subsection{Linear independence in $\Inv$}\label{sec:2.6}
We will need to be familiar with the following especially simple class of $\iota$-complexes:

\begin{definition}\label{def:2.9}
For $i \in \N$, define $X_i$ to be the $\iota$-complex generated over $\F[U]$ by three elements: $x$, $\iota x$, and $\alpha = \iota \alpha$. These have gradings given by $\gr(x) = \gr(\iota x) = 0$ and $\gr(\alpha) = -2i + 1$; the differential is defined by $\partial \alpha = U^i(x + \iota x)$. We likewise have the dual complex $X_i^\vee$, which is generated by $x^\vee$, $\iota x^\vee$, and $\alpha^\vee = \iota \alpha^\vee$. These have gradings given by $\gr(x^\vee) = \gr(\iota x^\vee) = 0$ and $\gr(\alpha^\vee) = 2i - 1$; the differential is defined by $\partial x^\vee = \partial \iota x^\vee = U^i \alpha^\vee$. The complexes $X_i$ and $X_i^\vee$ are displayed in Figure~\ref{fig:Xi}, along with their homologies. We will often also write $X_i$ for the local equivalence class of $X_i$ in $\Inv$, and similarly for $X_i^\vee$.
\end{definition}

\begin{figure}[h!]
\center
\includegraphics[scale=0.8]{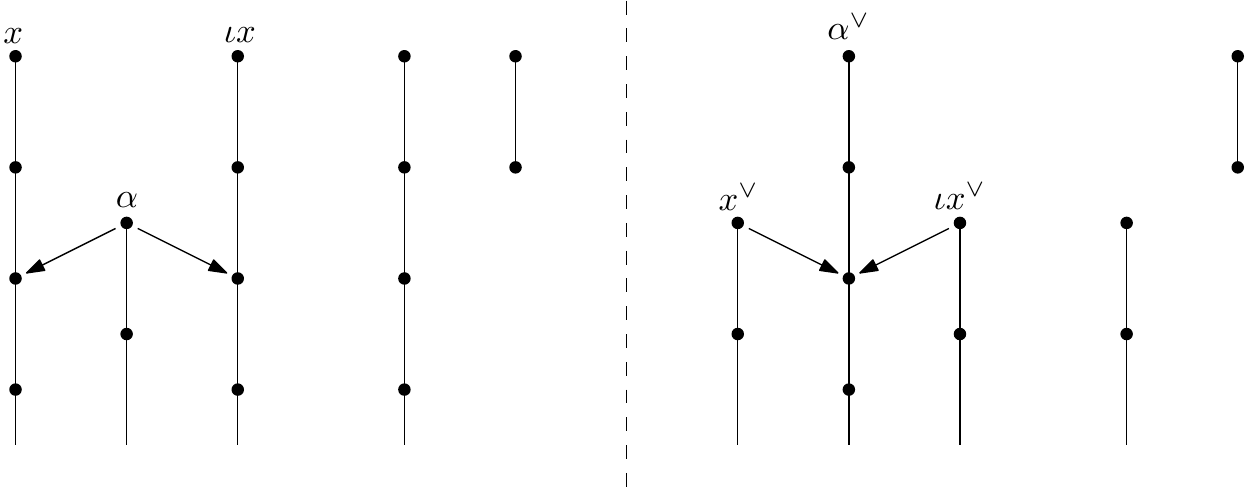}
\caption{Left: the complex $\smash{X_2}$ and its homology. Right: the complex $\smash{X_2^\vee}$ and its homology. In both cases, generators over $\F$ are represented by dots; the action of $U$ is given by following the vertical line segments downwards. For the two chain complexes, the action of $\partial$ is given by extending the indicated arrows $U$-equivariantly.}\label{fig:Xi}
\end{figure}

The $X_i$ turn out to be fundamental for understanding the structure of $\Inv$ and occur as the local equivalence classes of several families of homology spheres; see for example \cite{DaiManolescu}. We have the following basic fact regarding the $X_i$:

\begin{theorem}{\cite[Proof of Theorem 1.7]{DaiManolescu}}\label{thm:2.10}
For $i \in \N$, the classes $X_i$ are linearly independent in $\Inv$.
\end{theorem}

We will also need a generalization of Theorem~\ref{thm:2.10} which follows from the proof of \cite[Proof of Theorem 1.3]{DHM}. In order to state this, we recall some notation. If $(C, \iota)$ is an $\iota$-complex, then it follows from the first condition of Definition~\ref{def:2.8} that $H_*(C)$ is isomorphic to an $\F[U]$-module of the form $\F[U] \oplus (U\text{-torsion})$. The grading of the uppermost generator of the copy of $\F[U]$ is well-defined and gives the \text{$d$-invariant} $d(C)$. 

\begin{definition}\label{def:2.11}
If $(C_1, \iota_1)$ and $(C_2, \iota_2)$ are two $\iota$-complexes with $d(C_1) = d(C_2)$, then we write $(C_1, \iota_1) \leq (C_2, \iota_2)$ if there is a local map from $(C_1, \iota_1)$ to $(C_2, \iota_2)$.
\end{definition}

We will often suppress writing $\iota_1$ and $\iota_2$ in the inequality (and when discussing local equivalence). Note that if $C_1$ is locally equivalent to $C_2$, then automatically $d(C_1) = d(C_2)$. However, in general, a local map from $C_1$ to $C_2$ only guarantees $d(C_1) \leq d(C_2)$; hence Definition~\ref{def:2.11} is stronger than the presence of a local map.\footnote{This is essentially matter of notational convention. Elsewhere in the literature, it is often assumed that $d(C) = 0$ for convenience, in which case an inequality is indeed equivalent to the existence of a local map.} It turns out that Definition~\ref{def:2.11} defines a partial order on $\Inv$; for a discussion of the importance of Definition~\ref{def:2.11}, see \cite{DHSTcobord}.

\begin{theorem}{\cite[Proof of Theorem 1.3]{DHM}}\label{thm:2.12}
Let $\{C_i\}_{i \in \N}$ be a sequence of $\iota$-complexes. Suppose there exists a sequence $(n_i)_{i \in \N}$ such that $n_i \rightarrow \infty$ and $C_i \leq X_{n_i}$ for each $i$. Then $\{C_i\}_{i \in \N}$ has infinite rank in $\Inv$.
\end{theorem}

\begin{proof}
We sketch the proof for the convenience of the reader. It was shown in \cite[Lemma 7.11]{DHM} that for any local equivalence class $C_i$ with $C_i \leq X_{n_i}$, the connected homology $\Hc(mC_i)$ (see \cite{HHL} for a definition) has a $U$-torsion tower of length at least $n_i$ whenever $m \neq 0$. We then build a infinite linearly independent subsequence of the $C_{i}$ as follows. At the $p$-th stage, let $i_p$ be any integer for which $n_{i_{p}}$ is larger than the maximal $U$-torsion tower length appearing amongst $\Hc(C_{i_1})$, $\Hc(C_{i_2}), \cdots, \Hc(C_{i_{p-1}})$. It follows $\Hc(mC_{i_p})$ has a $U$-torsion tower of length at least $n_{i_{p}}$. According to \cite[Lemma 7.10]{DHM}, this implies that $mC_{i_{p}}$ does not lie in the the span of $C_{i_1}$, $C_{i_2}, \cdots, C_{i_{p-1}}$ for any $m \neq 0$. Hence $\{C_i\}_{i \in \N}$ has infinite rank in $\Inv$.
\end{proof}

For concreteness, in Definition~\ref{def:2.9} we have normalized the $X_i$ such that each $d(X_i) = 0$. However, Theorems~\ref{thm:2.10} and \ref{thm:2.12} hold more generally upon apply a grading shift to each $X_i$ or $X_{n_i}$. More precisely, for any sequence of integers $d_i$, the grading-shifted classes $X_i[d_i]$ are linearly independent. Similarly, if $\{C_i\}_{i \in \N}$ is a sequence of $\iota$-complexes with $C_i \leq X_{n_i}[d_i]$, then $\{C_i\}_{i \in \N}$ again has infinite rank in $\Inv$. Note that in the latter case, Definition~\ref{def:2.11} requires $d_i = - d(C_i)$; otherwise Theorem~\ref{thm:2.12} is false.\footnote{Our convention for the grading-shift notation is that an element of $C$ in grading zero has grading $-\Delta$ in $C[\Delta]$.} These minor extensions easily follow from considering the splitting
\[
\Inv = \Inv_0 \oplus \Z,
\]
where $\Inv_0$ is the subgroup of $\Inv$ consisting of all $\iota$-complexes with $d$-invariant zero.

\subsection{Involutive surgery formula}\label{sec:2.7}
Let $K$ be an oriented knot in $S^3$. In \cite{HM}, Hendricks and Manolescu defined a grading-preserving, skew-filtered homotopy involution 
\[
\iota_K \colon \CFK^\infty(K) \rightarrow \CFK^\infty(K)
\]
on the knot Floer complex of $K$. As in the 3-manifold case, the (filtered) homotopy equivalence class of the pair $(\CFK^\infty(K), \iota_K)$ is a diffeomorphism invariant of $K$. Although in general the action of $\iota_K$ is difficult to compute, there are a wide variety of cases in which $\iota_K$ is determined for formal reasons; these include thin knots and L-space knots (see \cite[Section 7]{HM} and \cite[Section 8]{HM}). 

In this paper, we will mainly utilize the involutive knot Floer package in the context of the involutive surgery formula. To state this, we first review some notation from the usual knot Floer surgery formula. For $s \in \mathbb{Z}$, let $A^{-}_{s}(K)$ be the subcomplex of $\CFK^\infty(K)$ given by:
\[
A^{-}_{s}(K) = \text{span}_{\ff[U]} \{ [\x,i,j], \; \textrm{such that} \; i \leq 0, j \leq s \}.
\]
Let $B^{-}_{s}(K)$ be the subcomplex of $\CFK^\infty(K)$ given by:
\[
B^{-}_{s}(K) = \text{span}_{\ff[U]}\{ [\x,i,j], \; \textrm{such that} \; i \leq 0 \}.
\]
Note that $H_*(B^{-}_{s}(K)) \cong \F[U]$ for any $s$. We also have the inclusion map 
\[
v: A^{-}_{s}(K) \rightarrow B^{-}_{s}(K).
\]
See \cite{OSknots, Rasmussen} for further discussion.

In \cite[Section 6]{HM}, Hendricks and Manolescu established a large surgery formula by showing that for any integer $p \geq g(K)$, there is a relatively graded homotopy equivalence
\[
(\CFm(S^3_p(K), [0]), \iota) \simeq (A^-_0(K), \iota_K).
\]
Note that $\iota_K$ preserves $A^-_0(K)$. In \cite{HHSZ} Hendricks, Hom, Zemke and the fourth author extended this to a general surgery formula for computing $\smash{(\CFm(S^3_{p/q}(K), [0]), \iota)}$. The local equivalence class of the resulting $\iota$-complex is easily described. The following will be the main technical tool used in this paper:

\begin{theorem}\cite[Proposition 22.9]{HHSZ}\label{thm:2.13}
Let $p$ and $q$ be positive, relatively prime integers. Suppose $p$ is odd, so that $[0]$ is the unique self-conjugate $\spinc$-structure on $\smash{S^{3}_{p/q}(K)}$. Then:
\begin{enumerate}
\item If $q$ is odd, $(\CFm(S^{3}_{p/q},[0]), \iota)$ is locally equivalent to $(A^{-}_{0}(K), \iota_K)$.
\item If $q$ is even, $(\CFm(S^{3}_{p/q},[0]), \iota)$ is locally equivalent to truncated mapping cone complex below:
\[\begin{tikzpicture}
   
\draw[-{Latex[round]}] (0,0) node [left] {$A^{-}_{[p/{2q}]}$} -- ++(1.4,-1.4) node [below] {$B^{-}_{[p/{2q}]}$};
\draw[-{Latex[round]}] (3,0) node [right] {$A^{-}_{[p/{2q}]}$} -- ++(-1.3,-1.4);

  \end{tikzpicture}\]
\end{enumerate}
Here, $[n]$ represents the integer closest to $n$.\footnote{No half-integers for $n$ appear in this article.} The $\iota$-action on this complex is given by interchanging the two copies of $A^{-}_{[p/2q]}$ and fixing $B^{-}_{[p/{2q}]}$.
\end{theorem}

Note that if $q$ is even, up to local equivalence the action of $\iota$ does not in fact depend on $\iota_K$.


\section{Proof of Theorems 1.7 and 1.9}\label{sec:3}

We now turn to the proof of Theorem~\ref{thm:1.9}, which will quickly imply Theorem~\ref{thm:1.7}. 

\subsection{Families of surgeries} \label{sec:3.1}
We begin with a general theorem regarding the linear independence of families of homology spheres obtained by odd-over-even surgeries on knots.

\begin{theorem}\label{thm:3.1}
Let $p$ and $q$ be positive integers with $p$ odd and $q$ even. Let $\{J_n\}_{n \in \N}$ be any family of knots in $S^3$. If $V_0(J_n) \rightarrow \infty$ as $n \rightarrow \infty$, then the family of $\Z_2$-homology spheres $\{\smash{S^3_{p/q}(J_n)}\}_{n \in \N}$ has infinite rank in $\smash{\Theta^{3}_{\Z_2}}$.
\end{theorem}

\begin{proof}
For convenience, write $s = [p/(2q)]$. Denote $A_s^- = A_s^-(J_n)$ and $B_s^- = B_s^-(J_n)$, so that the $\iota$-complex $C_n$ of $\smash{S^{3}_{p/q}(J_n)}$ is locally equivalent to the complex $A_s^- \oplus A_s^- \rightarrow B_s^-$ of Theorem~\ref{thm:2.13}. The structure of knot Floer homology implies we have a relatively graded isomorphism
\[
H_*(A_s^-) \cong \F[U] \oplus (U\text{-torsion}) \quad \text{and} \quad H_*(B_s^-) \cong \F[U].
\] 
The map $v$ induces an injection from the $\F[U]$-tower of $H_*(A_s^-)$ to $H_*(B_s^-) \cong \F[U]$ which is modeled on multiplication by some power of $U$.\footnote{The decomposition $H_*(A_s^-) \cong \F[U] \oplus (U\text{-torsion})$ is of course not canonical, but this statement holds for any choice of decomposition.} This power of $U$ is precisely the knot Floer concordance invariant $V_s(J_n)$ defined in \cite{NiWu}. 

Although in general it is difficult to completely understand $C_n$, we will show that the above discussion suffices to produce an inequality
\[
X_{V_s(J_n)}^\vee[-d(C_n)] \leq C_n. 
\]
We then dualize and apply Theorem~\ref{thm:2.12} This will give the linear independence of the $C_n^\vee$, and hence the $C_n$.

Let $a \in A_s^-$ be a cycle generating the $\F[U]$-tower in $H_*(A_s^-)$ and let $b \in B_s^-$ be a cycle generating $H_*(B_s^-) \cong \F[U]$. Write $a_1$ for the copy of $a$ in the first summand of $A_s^- \oplus A_s^-$ and $a_2$ for the copy of $a$ in the second. Note that $v(a)$ is homologous to $\smash{U^{V_s(J_n)} b}$; let $c \in B_s^-$ be such that $\smash{\partial c = v(a) + U^{V_s(J_n)} b}$. 

Define a (grading-homogenous) map from $\smash{X_{V_s(J_n)}^\vee}$ to $C_n$ by sending
\[
x^\vee \mapsto a_1 + c, \quad \iota x^\vee \mapsto a_2 + c, \quad \text{and} \quad \alpha^\vee \mapsto b
\]
This is an $\iota$-equivariant chain map; the situation is schematically depicted in Figure~\ref{fig:3.1}. The $d$-invariant of $C_n$ is given by the grading of $a_1 + a_2$. Applying a grading shift to $\smash{X_{V_s(J_n)}^\vee}$ so that $\smash{x^\vee}$ and $\smash{\iota x^\vee}$ have this grading then gives the claimed inequality.

\begin{figure}[h!]
\center
\includegraphics[scale=1]{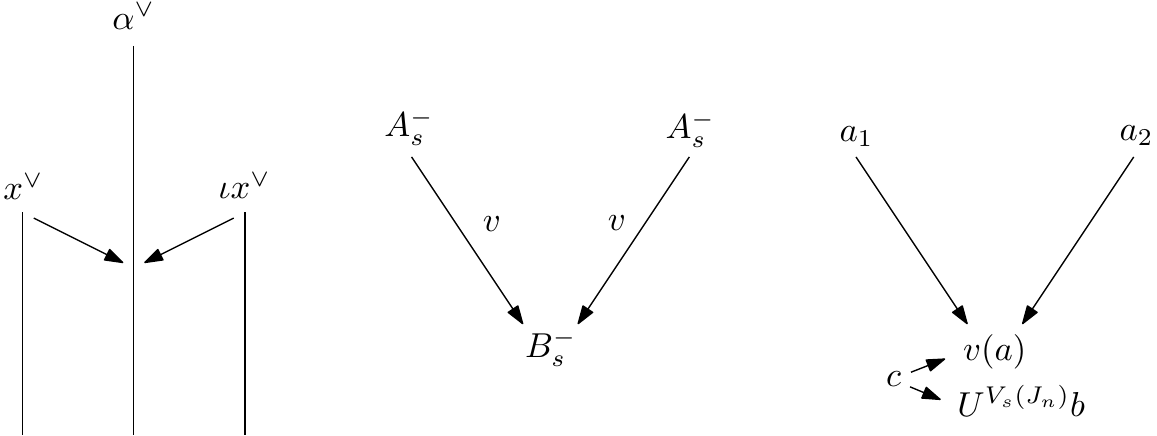}
\caption{Left: the complex $\smash{X_{V_s(J_n)}^\vee}$. Middle: the complex $A_s^- \oplus A_s^- \rightarrow B_s^-$ afforded by Theorem~\ref{thm:2.13}. Right: a subcomplex of $A_s^- \oplus A_s^- \rightarrow B_s^-$ isomorphic to $\smash{X_{V_s(J_n)}^\vee}$.}\label{fig:3.1}
\end{figure}

Dualizing, we obtain an inequality from $C_n^\vee$ to some grading shift of $\smash{X_{V_s(J_n)}}$. By \cite{Rasmussen_lens, NiWu}, we have that
\[
V_0(J_n) - s \leq V_s(J_n) \leq V_0(J_n).
\]
Since $s$ is independent of $n$, the condition $V_0(J_n) \rightarrow \infty$ implies that $V_s(J_n) \rightarrow \infty$. Theorem~\ref{thm:2.12} then shows that the span of the $C_n^\vee$, and hence the span of the $C_n$, has infinite rank in $\Inv$. 
\end{proof}

In fact, an examination of the proof of Theorem~\ref{thm:3.1} establishes a slightly stronger claim: we may even allow the surgery coefficient $p/q$ to vary, so long as $s = [p/(2q)]$ remains bounded. (This extension will not be used in the present paper.) Theorem~\ref{thm:3.1} immediately gives a general result regarding rational unknotting number one patterns. For completeness, we record this here:

\begin{theorem}\label{thm:3.2}
Let $P$ be a positive proper rational unknotting number one pattern with associated knot $J$. For any family of knots $\{K_n\}_{n \in \N}$ in $S^{3}$, let $\smash{J_n = J_{K_n, \tilde{\mu}}}$ be the corresponding doubly-infected family. If $V_{0}(J_n) \rightarrow \infty$ as $n \rightarrow \infty$, then $\{P(K_{n})\}_{n \in \N}$ has infinite rank.
\end{theorem}
\begin{proof}
As explained in Section~\ref{sec:2}, by passing to branched double covers and applying the discussion of Section~\ref{sec:2.3}, it suffices to prove that $\{\smash{S^{3}_{p/q}(J_n)}\}_{n \in \N}$ has infinite rank in $\smash{\Theta^{3}_{\Z_2}}$. As explained in Section~\ref{sec:2.2}, the fact that $P$ is proper means that $q$ is even, while $p$ is always odd. The claim then follows verbatim from Theorem~\ref{thm:3.1}.
\end{proof}

Theorem~\ref{thm:3.2} is trivially a specialization of Theorem~\ref{thm:3.1}: so far, we have not imposed any condition on $\ell = \lk(J, \tmu)$, nor have we used any details of the definition of a rational tangle pattern other than the fact that their branched double covers are surgeries on knots. However, while in principle Theorem~\ref{thm:3.2} is entirely general, in practice it may be difficult to check the condition $V_{0}(J_n) \rightarrow \infty$, since the knots $J_n$ are extremely complicated.

Our approach will thus be to estimate $V_0(J_n)$ in terms of the invariants of the companion knots $K_n$. We show that if $\ell \neq 0$, then we can bound $V_0(J_n)$ below in terms of $V_0(K_n) - V_0(-K_n)$, which will establish Theorem~\ref{thm:1.9}. The desired inequality will follow from the construction of a certain negative-definite cobordism whose incoming end is positive surgery on $J_n$. The outgoing end of our cobordism will be the connected sum of three pieces: positive surgery on $K_n$, negative surgery on $K_n$, and a third fixed manifold which is independent of $K_n$. A similar cobordism was considered in \cite{HPC}. We provide an elementary discussion of this technique in the next subsection.

\subsection{Construction of the cobordism}\label{sec:3.2}
Fix any non-zero integer $M$. In Figure~\ref{fig:3.2}, we have displayed an alternative surgery diagram for $M$-surgery on $\smash{J_{K, \tmu}}$. This consists of a copy of $J$ with surgery coefficient $M$, together with $K$, $\tau K$, $\mu$, and $\tau \mu$, all with surgery coefficient zero. For convenience, we denote these by $K_1$, $K_2$, $\mu_1$, and $\mu_2$, respectively. For ease of bookkeeping, we give $K_2$ and $\mu_2$ the reversed orientation as compared to their pushforward orientations under $\tau$. (Note that in the discussion of Section~\ref{sec:2.2}, both of these are given the pushforward orientation. Hence we may simultaneously reverse orientation on both at no cost.) Then
\[
\lk(\mu_i, K_i) = 1 \quad \text{and} \quad \lk(\mu_i, J) = \ell.
\]
for $i = 1, 2$.

To see that the diagram of Figure~\ref{fig:3.2} is correct, slide the strands of $J$ which pass though $\mu_i$ over $K_i$, via $J \mapsto J - \ell K_1 - \ell K_2$. This changes $J$ into $J_{K, \tmu}$ (with surgery coefficient $M$) and unlinks $J$ from each $\mu_i$. We then use $\mu_i$ to separate $K_i$ from the rest of the diagram and delete both pairs $\mu_i$ and $K_i$.

\begin{figure}[h!]
\center
\includegraphics[scale=0.8]{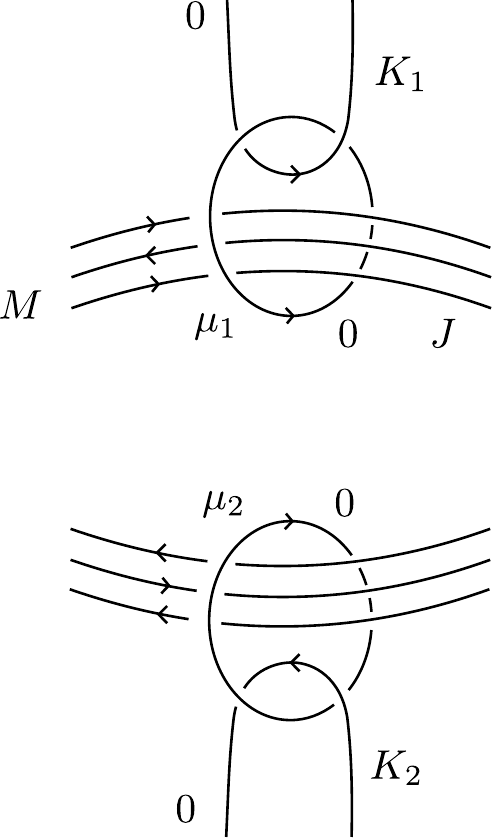}
\caption{A surgery diagram for $M$-surgery on $J_{K, \tmu}$. In this example, $\ell = 1$.}\label{fig:3.2}
\end{figure}

Now fix any pair of integers $N_1$ and $N_2$. Construct a cobordism $W$ from $M$-surgery on $J_{K, \tmu}$ by attaching a pair of $2$-handles along the curves $\gamma_1$ and $\gamma_2$ indicated on the left in Figure~\ref{fig:3.3}. These have framings $-N_1$ and $-N_2$, respectively. For concreteness, we orient $\gamma_1$ and $\gamma_2$ such that the non-zero linking numbers are given by
\[
\lk(\gamma_i, \mu_i) = 1 \quad \text{and} \quad \lk(\gamma_i, K_i) = -N_i.
\]
The outgoing end of this cobordism is homeomorphic to the connected sum
\[
S^3_{N_1}(K) \# Y(J, \mu, M) \# S^3_{N_2}(K),
\]
where $Y(J, \mu, M)$ is a 3-manifold which depends only on $J$, $\mu$, and $M$. (In particular, it does not depend on $K$.) To see this, slide $K_i$ over $\gamma_i$, as in Figure~\ref{fig:3.3}.

\begin{figure}[h!]
\center
\includegraphics[scale=0.8]{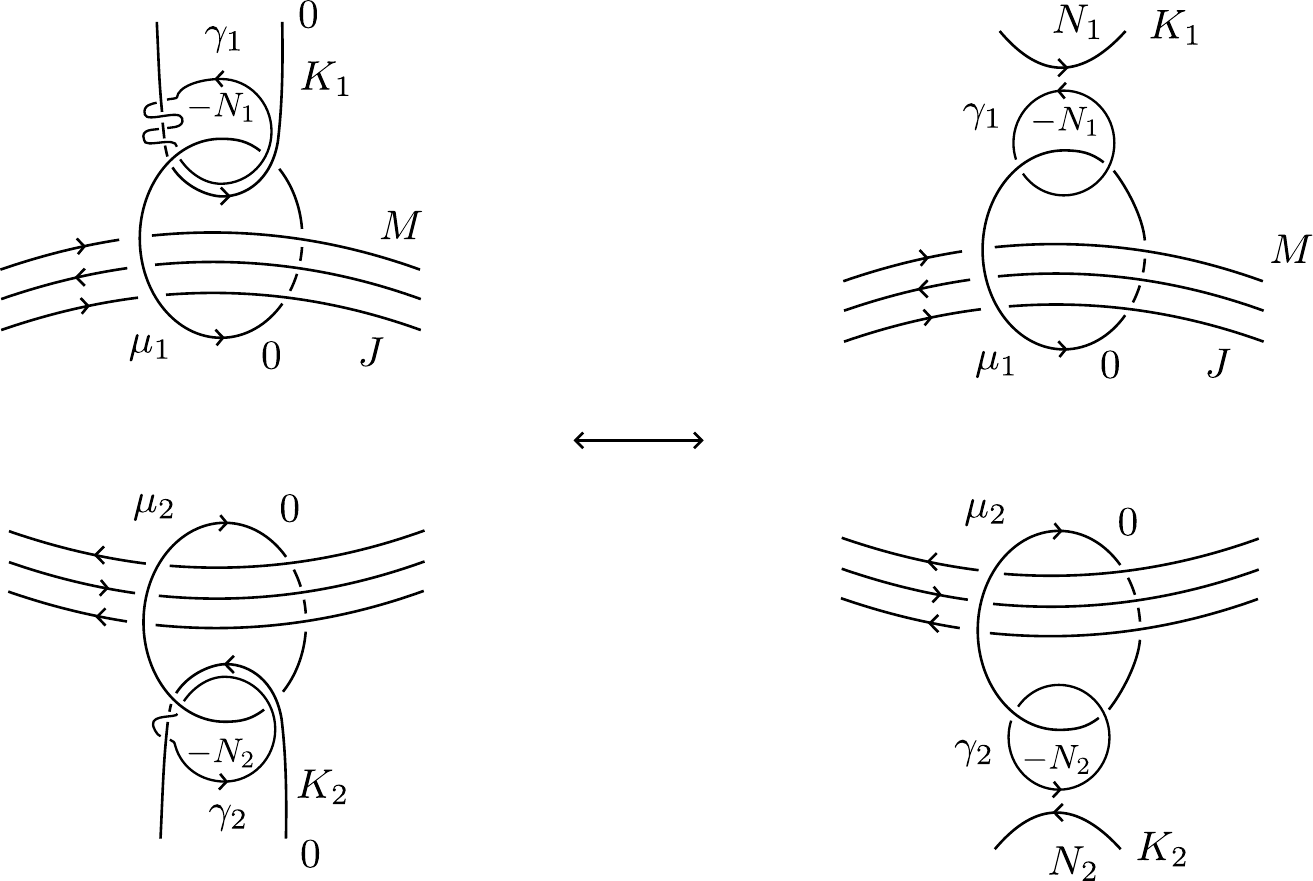}
\caption{Left: constructing a cobordism by attaching a pair of 2-handles along the curves $\gamma_1$ and $\gamma_2$. Right: the outgoing end of this cobordism. The equivalence between the left- and right-hand pictures is most easily envisioned by going from right-to-left, in which case the map is given by sliding $K_i$ over $\gamma_i$ via $K_i \mapsto K_i + \gamma_i$. The map from left-to-right corresponds to the slide $K_i \mapsto K_i - \gamma_i$. The 3-manifold $Y(J, \mu, M)$ is surgery on the sublink on the right consisting of $J$, $\mu_i$, and $\gamma_i$. In this example, $\ell = 1$, $N_1 = 2$, and $N_2 = -1$.}\label{fig:3.3}
\end{figure}

We now investigate under what conditions this cobordism is negative definite:

\begin{lemma}\label{lem:3.3}
The cobordism of Figure~\ref{fig:3.3} is negative definite if and only if 
\[
\ell^2(N_1^2 + N_2^2) - (N_1 + N_2)M  > 0 \quad \text{and} \quad MN_1N_2(- \ell^2(N_1 + N_2) + M) > 0.
\]
\end{lemma}
\begin{proof}
Since $M \neq 0$, the incoming end of the cobordism is a rational homology sphere with first homology $\Z/M\Z$. Over $\Q$, the second homology of the cobordism is thus clearly of rank two and is generated by the cores of the 2-handles attached along $\gamma_1$ and $\gamma_2$. However, in order to calculate the self-intersections of these generators, we must perform handleslides on the $\gamma_i$ to algebraically unlink them from the rest of the diagram. In addition, the $\gamma_i$ are not null-homologous in general; instead, we are only guaranteed that they are $M$-torsion in first homology.

We thus instead consider the curves
\begin{align*}
&C_1 = M\gamma_1 + (N_1 \ell^2 - M)K_1 + MN_1 \mu_1 - N_1 \ell J + N_1 \ell^2 K_2 \\
&C_2 = M\gamma_2 + (N_2 \ell^2 - M)K_2 + MN_2 \mu_2 - N_2 \ell J + N_2 \ell^2 K_1.
\end{align*}
This are obtained by sliding $M\gamma_1$ and $M\gamma_2$ over the other curves in the diagram to make them algebraically unlinked from the left-hand side of Figure~\ref{fig:3.3}. Indeed, the reader should verify that the linking numbers between $C_1$ and the five curves $K_1, \mu_1, J, \mu_2$, and $K_2$ are zero, and similarly for $C_2$. For convenience, we recall that the non-zero linking numbers are given by
\[
\lk(\mu_i, K_i) = \lk(\mu_i, \gamma_i) = 1, \quad \lk(\mu_i, J) = \ell, \quad \text{and} \quad \lk(\gamma_i, K_i) = -N_i,
\]
and
\[
\lk(\gamma_i, \gamma_i) = - N_i \quad \text{and} \quad \quad \lk(J, J) = M.
\]
The self-linking of $C_1$ is given by
\begin{align*}
lk(C_1, C_1) &= M lk(C_1, \gamma_1) \\
&= M\left( M(-N_1) + (N_1 \ell^2 - M)(-N_1) + MN_1 \right) \\
&= MN_1(M - N_1 \ell^2).
\end{align*}
Similarly, $\lk(C_2, C_2) = MN_2(M - N_2 \ell^2)$. Meanwhile, 
\[
\lk(C_1, C_2) = M \lk(C_1, \gamma_2) = - M N_1 N_2 \ell^2.
\]
Thus, the intersection form of the cobordism is proportional to
\[
\left(\begin{array}{cc}N_1(M - N_1 \ell^2) & -N_1N_2\ell^2 \\-N_1N_2\ell^2 & N_2(M - N_2 \ell^2)\end{array}\right).
\]
This has characteristic polynomial
\[
t^2 + \left( \ell^2(N_1^2 + N_2^2) - (N_1 + N_2)M \right) t + \left( MN_1N_2(- \ell^2(N_1 + N_2) + M) \right).
\]
The roots of the characteristic polynomial are simultaneously negative if and only if their sum is negative and their product is positive; that is,
\[
\ell^2(N_1^2 + N_2^2) - (N_1 + N_2)M  > 0 \quad \text{and} \quad MN_1N_2(- \ell^2(N_1 + N_2) + M) > 0,
\]
as desired.
\end{proof}

In our situation, we will be interested in large positive surgery on $J_{K, \tmu}$, since $V_0(J_{K, \tmu})$ is (up to an overall shift) given by the $d$-invariant of such a manifold. We thus assume that $M$ is positive. Suppose in addition that $\ell \neq 0$. Then we have:

\begin{lemma}\label{lem:3.4}
Let $M > 0$ and $\ell \neq 0$. Then for any $N_1 \gg 0$ and $N_2 < 0$ with $N_2$ small in magnitude compared to $N_1$, the cobordism $W$ is negative definite.
\end{lemma}
\begin{proof}
Under the hypotheses of the lemma, the conditions of Lemma~\ref{lem:3.3} are equivalent to
\[
N_1^2 + N_2^2 - \left(N_1 + N_2\right) \dfrac{M}{\ell^2} > 0 \quad \text{and} \quad N_1N_2\left( -(N_1 + N_2) + \dfrac{M}{\ell^2} \right) > 0.
\]
The first condition is clearly satisfied as long as at least one of $N_1$ and $N_2$ is sufficiently large in magnitude. In addition, if $N_1$ is sufficiently positive compared to the magnitude of $N_2$, then the factor $-(N_1 + N_2) + M/\ell^2$ in the second condition is negative. Since $N_1 > 0$ and $N_2 < 0$, this gives the claim.
\end{proof}

\begin{remark}\label{rem:3.5}
Let $P$ be any pattern with winding number zero. The meridian $\mu$ of $P$ lies on the boundary of the solid torus $S^1 \times D^2$ for $P$ and thus inherits a normal framing as a curve on $\partial(S^1 \times D^2)$. Lifting this normal framing to the branched double cover defines a pushoff of $\tmu$, which we denote by $\tmu'$. In \cite{HPC}, the authors consider the rational linking number $\lk(\tmu, \tmu')$ and impose the condition $\lk(\tmu, \tmu') \neq 0$ as a hypothesis of \cite[Theorem 3]{HPC}. In our context, $\tmu'$ may be obtained by taking the Seifert framing of $\tmu$ before surgering along $J$ in Figure~\ref{fig:3.3}; the quantity $\lk(\tmu, \tmu')$ is the rational linking number of $\tmu$ and $\tmu'$ in the surgered manifold. It is easily checked that this is non-zero if and only if $\ell = \lk(\tmu, J)$ is non-zero. Thus the linking number requirement we impose in this paper is the same as that of \cite{HPC}; moreover, this characterization shows that the condition $\ell \neq 0$ depends only on $P$ (and not the choice of unknotting tangle replacement).
\end{remark}

\subsection{Completion of proof} We now finally conclude the proof of Theorem~\ref{thm:1.9}. We recall the statement for the convenience of the reader: \\
\\
\noindent
\textbf{Theorem 1.9.} \textit{Let $P$ be a proper rational unknotting number one pattern with non-zero linking number and $p/q > 0$. If $\{K_n\}_{n \in \N}$ is any family of knots such that $V_0(K_n) - V_0(-K_n) \rightarrow \infty$ as $n \rightarrow \infty$, then $\{P(K_{n})\}_{n \in \N}$ has infinite rank.}

\begin{proof}
Fix any positive integer $M$ and integers $N_1$ and $N_2$ satisfying the conditions of Lemma~\ref{lem:3.4}. It will be useful to assume that $M$, $N_1$, and $N_2$ are odd. We obtain a negative-definite cobordism $W$ from
\[
S^3_M(J_n) = S^3_M(J_{K_n, \tmu})
\]
to
\[
S^3_{N_1}(K_n) \# Y \# S^3_{N_2}(K_n),
\]
where $Y = Y(J, \mu, M)$ is independent of $K_n$. We claim that there exists a $\spinc$-structure $\s$ on $W$ which restricts to $[0]$ on $S^3_M(J_n)$ and $[0]$ on both factors $S^3_{N_1}(K_n)$ and $S^3_{N_2}(K_n)$. This can be shown in many ways. For example, note that since $M$, $N_1$, and $N_2$ are odd, the $\spinc$-structures on $S^3_M(J_n)$, $S^3_{N_1}(K_n)$, and $S^3_{N_2}(K_n)$ are parameterized by the Chern classes of their determinant line bundles. Let $\mathfrak{t}$ be any $\spinc$-structure on $W$ with determinant line bundle $L$. Let
\[
2E + 1= MN_1N_2.
\]
Then the tensor product $\s = \mathfrak{t} \otimes L^E$ is a $\spinc$-structure on $W$ with first Chern class $(2E + 1) c_1(L)$; this trivially vanishes when restricted to $S^3_M(J_n)$, $S^3_{N_1}(K)$, and $S^3_{N_2}(K)$. We thus obtain an equality of $d$-invariants:
\begin{equation}\label{eq:eq1}
d(S^3_M(J_n), [0]) + \Delta(W, \s) \leq d(S^3_{N_1}(K_n), [0]) + d(Y, \s|_Y) + d(S^3_{N_2}(K_n), [0]).
\end{equation}
Here, $\Delta(W, \s)$ is the Heegaard Floer grading shift associated to $W$ and $\s$. Crucially, note that $\Delta(W, \s)$ and $d(Y, \s|_Y)$ do not depend on the index $n$.

Now, since $M$ and $N_1$ are positive, we have the standard equality
\[
d(S_M(J_n), [0]) = \dfrac{M-1}{4} - 2V_0(J_n) \quad \text{and} \quad d(S_{N_1}(K), [0]) = \dfrac{N_1-1}{4} - 2V_0(K_n).
\]
Since $N_2$ is negative, we have
\[
d(S_{N_2}(K_n), [0]) = - d(S_{-N_2}(-K_n), [0]) = \dfrac{N_2 + 1}{4} + 2V_0(-K_n).
\]
Substituting these into our inequality (\ref{eq:eq1}) for $d$-invariants and collecting terms, we obtain
\[
V_0(J_n) \geq V_0(K_n) - V_0(-K_n) + C
\]
where $C$ is a constant not depending on $n$. Hence we see that the condition $V_0(K_n) - V_0(-K_n) \rightarrow \infty$ in fact guarantees $V_0(J_n) \rightarrow \infty$. Applying Theorem~\ref{thm:3.2} then gives the result.
\end{proof}

\begin{remark}\label{rem:3.6}
The reader may wonder whether the condition $\ell \neq 0$ is necessary. This is crucial for the argument: note that if $\ell = 0$, then the conditions of Lemma~\ref{lem:3.3} become
\[
-(N_1 + N_2)M > 0 \quad \text{and} \quad M^2N_1N_2 > 0.
\]
If $M > 0$, these conditions are only satisfied when $N_1$ and $N_2$ are both less than zero. In this case, the resulting inequality bounds $V_0(J_n)$ below by a constant plus $-2V_0(-K_n)$, which is not generally useful (as $V_0$ is positive). Similarly, the reader may wonder whether more judicious choices for $N_1$ and $N_2$ might produce different inequalities. For example, if we could choose $N_1$ and $N_2$ to both be positive, we would bound $V_0(J_n)$ below by a constant plus $2 V_0(K_n)$. Unfortunately, this is also impossible: if $N_1$ and $N_2$ are positive, the conditions of Lemma~\ref{lem:3.3} become
\[
N_1^2 + N_2^2 - \left(N_1 + N_2\right) \dfrac{M}{\ell^2} > 0 \quad \text{and} \quad -(N_1 + N_2) + \dfrac{M}{\ell^2} > 0.
\]
It is straightforward to verify that this is impossible.
\end{remark}

This immediately completes the proof of Theorem~\ref{thm:1.7}: \\
\\
\noindent
\textbf{Theorem 1.7.} \textit{Let $P$ be a proper rational unknotting number one pattern with non-zero linking number. Then $P$ is rank-expanding. More specifically, if $K$ is any knot such that $V_0(nK) - V_0(-nK) \rightarrow \infty$ as $n \rightarrow \infty$, then $P$ is rank-expanding along $K$.}

\begin{proof}
Let $P$ be a proper rational unknotting number one pattern with $\ell \neq 0$. If $P$ is positive, then setting $K_n = nK$ and applying Theorem~\ref{thm:1.9} immediately gives the claim. Otherwise, consider the (mirrored, orientation-reversed) pattern $-P$. This is also a proper rational unknotting number one pattern with $\ell \neq 0$; considering the branched double cover shows that $-P$ is positive. Again setting $K_n = nK$, Theorem~\ref{thm:1.9} implies $(-P)(nK)$ (for $n > 0$) has infinite rank. But this means $P(-nK)$ (for $n > 0$) has infinite rank. Hence $P$ is certainly rank-expanding along $\{nK\}_{n \in \Z}$.
\end{proof}

\section{Proof of Theorem 1.10}\label{sec:4}
We now turn to the proof of Theorem~\ref{thm:1.10}. \\
\\
\noindent
\textbf{Theorem 1.10.} 
\textit{
Let $P$ be a $p/q$-rational tangle pattern with $p/q > 0$.
\begin{enumerate}
\item Suppose $q$ is even. Let $\{K_n\}_{n \in \N}$ be any family of thin knots with $\tau(K_n)$ distinct and greater than $\lfloor(\lfloor p/q \rfloor + 1)/4 \rfloor$. Then $\{P(K_{n})\}_{n \in \N}$ is linearly independent and in fact spans a $\Z^\infty$-summand of $\cC$.
\item Suppose $q$ is odd.  Let $\{K_n\}_{n \in \N}$ be any family of thin knots with $\tau(K_n)$ distinct and less than zero. Then $\{P(K_{n})\}_{n \in \N}$ is linearly independent and in fact spans a $\Z^\infty$-summand of $\cC$.
\end{enumerate}
}

\begin{proof}
Let $P$ be a $p/q$-rational tangle pattern with $p/q > 0$. We start by showing that the surgered family $\{\smash{S^{3}_{p/q}(K_{n} \# K_{n})}\}_{n \in \N}$ is linearly independent in $\smash{\Theta^{3}_{\Z_2}}$.

We first consider the case when $q$ is even. Following the proof of Theorem~\ref{thm:1.7}, we again analyze the surgery complex afforded by Theorem~\ref{thm:2.13} and compare it to some $X_i^\vee$. However, because the knot Floer homology of a thin knot is very simple, in this case we will be able to establish an explicit local equivalence, rather than just an inequality. This will allow us to utilize Theorem~\ref{thm:2.10} rather than Theorem~\ref{thm:2.12}, and thus conclude linear independence.

As before, write $s = [p/2q]$ and denote $A_s^- = A_s^-(K_n \# K_n)$ and $B_s^- = B_s^-(K_n \# K_n)$, so that the $\iota$-complex of $\smash{S^{3}_{p/q}(K_{n} \# K_{n})}$ is locally equivalent to the complex $A_s^- \oplus A_s^- \rightarrow B_s^-$ defined in Theorem~\ref{thm:2.13}. 

It is a standard fact that if $K_n$ is thin, then the connected sum $K_n \# K_n$ is also thin. Hence the knot Floer complex of $K_n \# K_n$ consists of a step-length-one staircase, together with a number of side-length-one boxes, as schematically displayed on the left in Figure~\ref{fig:4.1}. The staircase has total height $2 \tau(K_n)$. The fact that $\tau(K_n) > 0$ (together with the fact that $\tau(K_n \# K_n) = 2\tau(K_n)$ is even) shows that the staircase opens towards the south-west, as in Figure~\ref{fig:4.1}.

\begin{figure}[h!]
\center
\includegraphics[scale=1]{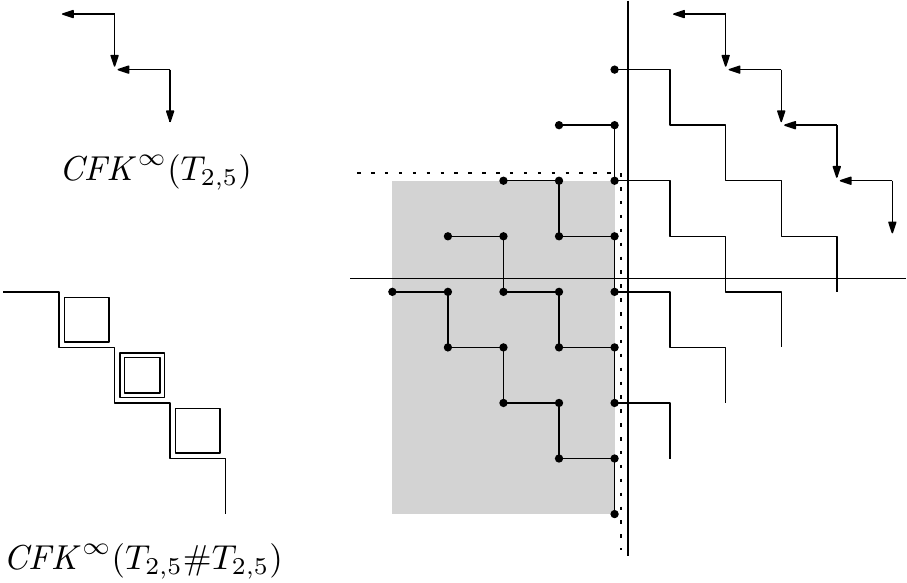}
\caption{Top left: the knot Floer complex of the thin knot $T_{2, 5}$; $\tau(T_{2,5}) = 2$. Bottom left: the knot Floer complex of $T_{2,5} \# T_{2,5}$ (after a change of basis). Right: the subcomplex of $\CFK^\infty(T_{2,5} \# T_{2,5})$ spanned by the staircase generators. We have schematically depicted $(A_s^-)'$ and $(B_s^-)'$ for $s = 2$. The generators in $(B_s^-)'$ are drawn as dots (i.e., to the left of the vertical axis); the generators in $(A_s^-)'$ are the subset of these dots lying in the shaded region.}\label{fig:4.1}
\end{figure}

We argue that up to local equivalence, we can successively simplify $\CFK^\infty(K_n \# K_n)$ until it is the same as some $\smash{X_i^\vee}$. The first simplification is as follows. Consider the subcomplex of $\CFK^\infty(K_n \# K_n)$ spanned by the staircase generators. Let the intersection of this subcomplex with $A_s^-$ be denoted $(A_s^-)'$, and define $(B_s^-)'$ similarly. (See the right of Figure~\ref{fig:4.1}.) This gives an obvious subcomplex $(A_s^-)' \oplus (A_s^-)' \rightarrow (B_s^-)'$ of $A_s^- \oplus A_s^- \rightarrow B_s^-$. The inclusion and projection maps for this subcomplex are easily checked to be local equivalences, so without loss of generality we may replace the original complex $A_s^- \oplus A_s^- \rightarrow B_s^-$ with $(A_s^-)' \oplus (A_s^-)' \rightarrow (B_s^-)'$. 

An examination of Figure~\ref{fig:4.1} shows that $H_*((A_s^-)')$ and $H_*((B_s^-)')$ are both copies of $\F[U]$. Thus, the associated graded complex 
\[
H_*((A_s^-)') \oplus H_*((A_s^-)') \rightarrow H^*((B_s^-)')
\]
with the induced map $v_* \oplus v_*$ is certainly isomorphic to a grading-shifted copy of $\smash{X_{V_s(K_n \# K_n)}^\vee}$. It is moreover easy to show that in this case, the associated graded complex is $\iota$-equivariantly homotopy equivalent to the original. Hence we obain the desired local equivalence
\[
X_{V_s(K_n \# K_n)}^\vee[-d(C_n)] \simeq C_n. 
\]

An examination of Figure~\ref{fig:4.1} shows
\[
V_s(K_n \# K_n) = \lceil (\tau(K_n \# K_n) - s)/2 \rceil = \tau(K_n) - \lfloor s/2 \rfloor
\] 
so long as the right-hand side is positive, and $V_s(K_n \# K_n) = 0$ otherwise. (Consider the copy of the staircase on the right of Figure~\ref{fig:4.1} which intersects the left-half plane in a single dot. Then $V_s(K_n \# K_n)$ is the number of diagonal translations needed for this staircase to intersect the shaded region.) Some numerological casework shows that 
\[
\lfloor s/2 \rfloor = \lfloor [p/(2q)]/2 \rfloor = \lfloor(\lfloor p/q \rfloor + 1)/4 \rfloor.
\]
The hypotheses of the theorem thus imply that the $\iota$-complexes of the $\smash{S^{3}_{p/q}(K_{n} \# K_{n})}$ are locally equivalent to grading-shifted copies of $X_i^\vee$, with $i$ positive and distinct. As these are linearly independent in $\Inv$, this completes the proof.

We now turn to the case when $q$ is odd. By Theorem~\ref{thm:2.13}, the $\iota$-complex of $\smash{S^{3}_{p/q}(K_{n} \# K_{n})}$ is locally equivalent to the large surgery complex $(A_0^-(K_n \# K_n), \iota_K)$. We attempt to understand $A_0^-(K_n \# K_n)$ explicitly. Much of this computation follows from \cite[Section 8]{HM}, so we will be brief.

In \cite[Section 8]{HM}, Hendricks and Manolescu calculate the $\iota_K$-complex of all thin knots. Their result shows that up to local equivalence, the $\iota_K$-complex of a thin knot is locally equivalent either to a staircase or a staircase plus a single side-length-one box. These possibilities are displayed in Figure~\ref{fig:4.2}; note that we now assume $\tau(K_n \# K_n)$ is negative. In the former case, the action of $\iota_K$ is the obvious reflection map on the staircase generators. In the latter, we have the slight modification (in the notation of Figure~\ref{fig:4.2}):
\[
\iota_K d = d + b, \quad \iota_K b = b + e, \quad \iota_K c = c' + a', \quad \text{and} \quad \iota_K c' = c + a,
\]
with $\iota_K$ acting by reflection on all other generators. We abuse notation and write $A_0^-$ for this simplified representative of the local equivalence class of $A_0^-(K_n \# K_n)$.

\begin{figure}[h!]
\center
\includegraphics[scale=1.2]{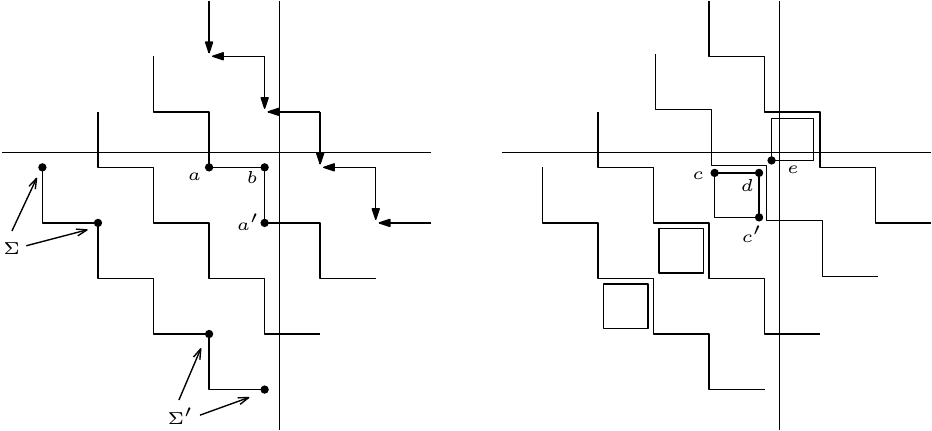}
\caption{Left: staircase with no box, with generators $a$, $a'$ and $b$ labeled. We have also labeled sums-of-generators $\Sigma$ and $\Sigma'$. To define these, consider the first copy of the staircase contained in the lower-left quadrant. Note that there are an odd number of non-cycle generators in this staircase. Let $\Sigma$ be the sum of such generators in the (strict) upper-half of this staircase and $\Sigma'$ be the reflection of $\Sigma$. Right: staircase with a single box; several further generators labeled. (The generators labels on the left are meant to carry over in the obvious way.)}\label{fig:4.2}
\end{figure}

We show that for a staircase with no box, $(A_0^-, \iota_K)$ is homotopy equivalent to (a grading-shifted copy of) $\smash{X_{|\tau(K_n)|}^\vee}$. For this, consider the subcomplex $S$ of $A_0^-$ spanned over $\F[U]$ by $a$, $a'$, and $b$, together with the sums-of-generators $\Sigma$ and $\Sigma'$. The reader may check that this is a $\iota_K$-equivariant subcomplex of $A_0^-$ which is homotopy equivalent to the original. Moreover, we claim that $S$ is homotopy equivalent to the complex on the left in Figure~\ref{fig:4.3}. One direction of this homotopy equivalence is given by the map
\[
f(x_0) = \Sigma + U^{|\tau(K_n)|} b + \Sigma', \quad f(x_1) = \Sigma, \quad \text{and} \quad f(x_2) = a.
\]
This does not intertwine $\iota_K$ with the $\iota$-action in Figure~\ref{fig:4.3} on the nose, but if we set
\[
H(x_0) = 0, \quad H(x_1) = 0, \quad \text{and} \quad H(x_2) = b
\]
then $f\iota + \iota_Kf = \partial H + H \partial$. We leave it to the reader to produce the homotopy equivalence in the other direction. A quick change-of-basis shows that up to grading shift, the left-hand complex in Figure~\ref{fig:4.3} is precisely $\smash{X_{|\tau(K_n)|}^\vee}$, giving the claim. For further discussion, see \cite[Example 2.6]{DHSTcobord}.

\begin{figure}[h!]
\center
\includegraphics[scale=1.3]{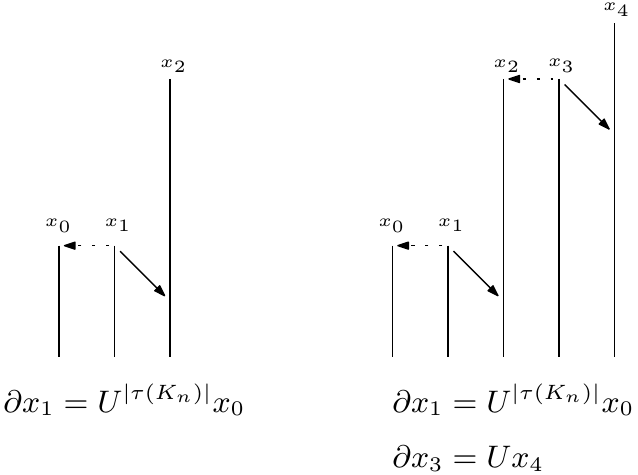}
\caption{Left: an $\iota$-complex with three generators. The dashed arrow represents the action of $\omega = 1 + \iota$; here $\omega x_1 = x_0$ and otherwise vanishes. The solid arrow represents $\partial$; here $\partial x_1 = U^{|\tau(K_n)|} x_2$ and otherwise vanishes. Right: an $\iota$-complex with five generators. Here $\omega x_1 = x_0$ and $\omega x_3 = x_2$. The differential is given by $\partial x_1 = U^{|\tau(K_n)|} x_2$ and $\partial x_3 = U x_4$.}\label{fig:4.3}
\end{figure}

We now turn to understanding the case of a staircase with box. In this case, it turns out that the $(A_0^-, \iota_K)$ is not locally equivalent to a copy of $\smash{X_i^\vee}$. However, it is still possible to understand its local equivalence class. To see, this we modify our subcomplex $S$ from before by additionally including the generators $c$, $c'$, $d$, and $e$, as displayed on the right in Figure~\ref{fig:4.3}. Once again, the reader can check that $S$ is a subcomplex of $A_0^-$ which is homotopy equivalent to the original. We further claim that $S$ is homotopy equivalent to the complex on the right in Figure~\ref{fig:4.3}. To see this, map
\[
f(x_0) = \Sigma + U^{|\tau(K_n)|} b + \Sigma', \quad f(x_1) = \Sigma, \quad f(x_2) = a, \quad f(x_3) = c', \quad \text{and} \quad f(x_4) = e.
\]
Setting
\[
H(x_0) = U^{|\tau(K_n)|-1} c, \quad H(x_1) = 0, \quad H(x_2) = b, \quad H(x_3) = d, \quad \text{and} \quad H(x_4) = 0,
\]
exhaustive evaluation on each generator gives $f\iota + \iota_Kf = \partial H + H \partial$. We leave it to the reader to produce the map in the other direction.

After a change-of-basis, the complex on the right of Figure~\ref{fig:4.3} is locally equivalent (up to grading shift) to
\[
X_{|\tau(K_n)|}^\vee \otimes X_1^\vee.
\]
The relevant computation follows from \cite[Lemma 5.2]{DaiStoffregen}; see also \cite[Theorem 8.1]{DHSTcobord}. Thus, for each $n$, we have that the $\iota$-complex of $\smash{S^{3}_{p/q}(K_{n} \# K_{n})}$ is locally equivalent to a grading-shifted copy of either $\smash{X_{|\tau(K_n)|}^\vee}$ or $\smash{X_{|\tau(K_n)|}^\vee \otimes X_1^\vee}$. Since the $\tau(K_n)$ are distinct, this gives the claim. 

Finally, we may upgrade the statement of linear independence to the spanning of a $\Z^\infty$-summand using the work of \cite{DHSTcobord}. In \cite{DHSTcobord}, the authors construct an infinite family of linearly independent homomorphisms from $\Theta^3_{\Z_2}$ to $\Z$, factoring through the homomorphism $h \colon \Theta^3_{\Z_2} \rightarrow \Inv$. More precisely, they construct an algebraically-defined auxiliary group $\smash{\widehat{\Inv}}$ with a homomorphism $\smash{\widehat{h} \colon \Inv \rightarrow \widehat{\Inv}}$ and define a family of linearly independent homomorphisms 
\[
\smash{\phi_n \colon \widehat{\Inv} \rightarrow \Z}. 
\]
Composing everything with the double branched cover homomorphism, we obtain a linearly independent family of homomorphisms
\[
\cC \xrightarrow{\Sigma_2} \Theta^3_{\Z_2} \xrightarrow{h} \Inv \xrightarrow{\widehat{h}} \widehat{\Inv} \xrightarrow{\phi_n} \Z.
\]
Moreover, in \cite{DHSTcobord} it is shown that $\phi_i(\smash{\widehat{h}}(X_j)) = \delta_{ij}$, where $\delta_{ij}$ is the Kronecker delta. This suffices to establish the claim.
\end{proof}

\section{Thin knots and L-space knots}\label{sec:5}

We now turn to some applications and examples of Theorems~\ref{thm:1.7} and \ref{thm:1.9}. The first order of business is to understand the general condition $V_0(K_n) - V_0(-K_n) \rightarrow \infty$ of Theorem~\ref{thm:1.9}. In practice, $V_0(-K_n)$ may often be known to be bounded, or even zero; for example, if $K$ is a positive L-space knot, or $K$ is a thin knot with $\tau(K) > 0$. Moreover, if $K_1$ and $K_2$ have $V_0(-K_1) = V_0(-K_2) = 0$, then by the sub-additivity of $V_0$ their connected sum has this property also. Thus, if we assume that the family $\{K_n\}_{n \in \N}$ is drawn from the monoid generated by positive L-space knots, or the monoid of thin knots with $\tau(K) > 0$, then the hypothesis of Theorem~\ref{thm:1.9} simplifies to $V_0(K_n) \rightarrow \infty$.

It is also natural to search for related conditions which do not explicitly reference any Floer-theoretic invariants. For the class of thin knots, this is straightforward: if $K$ is thin, then $V_0(K) = \max\{0, \lceil \tau(K)/2 \rceil\}$. Moreover, for all known examples of thin knots (including all alternating and quasi-alternating knots), we have $\tau(K) = - \sigma(K)/2$. In this case we may thus re-write the hypothesis of Theorem~\ref{thm:1.9} as $\sigma(K_n) \rightarrow - \infty$.

For L-space knots, finding a topological condition is slightly more involved. We introduce the following two (rather trivial) lemmas: 

\begin{lemma}\label{lem:5.1}
Let $K$ be a positive L-space knot and $n$ be the number of non-zero terms in the Alexander polynomial $\Delta_K$. Then 
\[
V_0(K) \geq \lfloor (n-1)/4 \rfloor.
\]
\end{lemma}
\begin{proof}
As is well-known to experts in Floer theory, it is straightforward to determine the knot Floer complex of $K$ (and thus the value of $V_0$) from $\Delta_K$ in the case that $K$ is an L-space knot \cite{OSlens}. Let
\[
\Delta_K = (-1)^m + \sum_{i = 1}^m (-1)^{m-i} (t^{n_i} + t^{-n_i})
\]
for $0 < n_1 < n_2 < \cdots < n_m$. Then (as is recorded in \cite[Section 7]{HM}), 
\[
V_0(K) = n_m - n_{m-1} + \cdots + (-1)^{m-2}n_2 + (-1)^{m-1}n_1.
\]
Since each pair $n_k - n_{k-1}$ is at least one, we of course have $V_0(K) \geq \lfloor m/2 \rfloor = \lfloor (n-1)/4 \rfloor$.
\end{proof}

\begin{lemma}\label{lem:5.2}
Let $K_1, \ldots, K_n$ be any collection of positive L-space knots. Then 
\[
V_0(K_1 \# \cdots \# K_n) \geq \max\{ \left[ n/2 \right], V_0(K_1), \ldots, V_0(K_n) \}.
\]
\end{lemma}

\begin{proof}
It is easy to check that 
\[
V_0(K_1 \# \cdots \# K_n) \geq V_0(K_i)
\]
for each $i$. Indeed, by sub-additivity, 
\[
V_0(K_i) - V_0(-K_1 \# \cdots \# -K_{i-1} \# -K_{i+1} \# \cdots \# -K_n) \leq V_0(K_1 \# \cdots \# K_n),
\]
but the second term is zero. The claim that $V_0(K_1 \# \cdots \# K_n)$ is at least $\lfloor n/2 \rfloor$ follows from work of Borodzik and Livingston \cite{Borodzik}. Explicitly, in \cite{Borodzik} it is shown how to calculate the $V_0$-invariant of a connected sum of positive L-space knots. We briefly recall their formulation. For each $l$, place the staircase complex $C_l$ corresponding to $K_l$ in the first quadrant. Let $S$ represent the set of all generators of the form $a_1 \otimes a_2 \otimes \cdots \otimes a_n \in C_1 \otimes C_2 \otimes \cdots \otimes C_n$, where $a_l$ vary among the generators of $C_l$ with Maslov grading zero. Let $\underline{S}$ represent the set pairs of integers obtained from the $(i,j)$-grading of the generators in $S$. From \cite[Proposition 5.1]{Borodzik} we have:
\[
V_0(K_1 \# K_2 \# \cdots \# K_n)= {\underset{(\alpha,\beta) \in \underline{S}}{\mathrm{min}}} \{ \mathrm{max} (\alpha, \beta) \}.
\]
Here, $\mathrm{max} (\alpha, \beta)$ represents the maximum of the two coordinates. Let us now take an arbitrary element $a_1 \otimes a_2 \otimes \cdots \otimes a_{n} \in S$ and assume that the $(i,j)$-coordinates of $a_l$ is $(x_l,y_l)$. Lastly, let $k$ be the number of times the $i$-coordinate of $(x_l,y_l)$ is $0$. Now observe that
\[
\sum_{l=1}^{n} x_{l} \geq n-k \quad \mathrm{and} \quad \sum_{l=1}^{n} y_{l} \geq k.
\]
Hence, we obtain
\[
\mathrm{max}\left\{ \sum_{l=1}^{n} x_{l}, \sum_{l=1}^{n} y_{l} \right\} \geq \left[ \frac{n}{2} \right].
\]
Since the choice of the element $a_1 \otimes a_2 \otimes \cdots \otimes a_{n} \in S$ was arbitrary, we obtain,
\[ 
V_0(K_1 \# K_2 \# \cdots \# K_n) \geq \left[ n/2 \right].
\]
completing the proof. 
\end{proof}

Now suppose the family $\{K_n\}_{n \in \N}$ in Theorem~\ref{thm:1.9} is drawn from the monoid generated by positive L-space knots. By Lemmas~\ref{lem:5.1} and \ref{lem:5.2}, it follows that $\{P(K_n)\}_{n \in \N}$ has infinite rank so long as either:
\begin{enumerate}
\item\label{alt:a} The number of summands in $K_n$ is unbounded as $n \rightarrow \infty$; or,
\item\label{alt:b} The set of Alexander polynomial lengths (that is, number of non-zero terms in each Alexander polynomial) occurring amongst summands of the $K_n$ is unbounded as $n \rightarrow \infty$.
\end{enumerate}
This leads to the following:

\begin{corollary}\label{cor:5.3}
Let $P$ be a proper rational unknotting number one pattern with $\ell \neq 0$. Let $\mathcal{M}$ be the monoid of positive linear combinations of (right-handed) torus knots. If $\{K_n\}_{n \in \N}$ is any infinite subset of $\mathcal{M}$, then $\{P(K_{n})\}_{n \in \N}$ has infinite rank.
\end{corollary}
\begin{proof}
Since $\mathcal{M}$ is a submonoid of the monoid generated by positive L-space knots, it suffices to show that any infinite family of elements of $\mathcal{M}$ must satisfy either (\ref{alt:a}) or (\ref{alt:b}) above. Suppose $\{K_n\}_{n \in \N}$ does not satisfy (\ref{alt:a}). Then the set of distinct individual torus knots $T_{p, q}$ that appear as summands in the $K_n$ must be infinite, and in particular have unbounded indices. (That is, either $p$ is unbounded or $q$ is unbounded.)

In \cite{Song}, it is shown that the number of non-zero terms in $\smash{\Delta_{T_{p,q}}}$ is given by $vx + uy$, where $x$, $y$, $u$, and $v$ are positive integers such that $vx - uy = 1$, $p = x + y$, and $q = u + v$. Since $u$ and $v$ are at least one, we certainly have $vx + uy \geq x + y = p$. Similarly, $vx + uy \geq v + u = q$. Hence any infinite family of torus knots has unbounded Alexander polynomial length (as measured by number of non-zero terms).
\end{proof}

For the patterns considered in \cite{HPC}, it follows from the proof \cite[Theorem 3]{HPC} that any infinite family of distinct torus knots has infinite-rank image. In Corollary~\ref{cor:5.3}, we extend this by allowing the family of companion knots to be drawn from sums of torus knots, rather than individual torus knots. 

We now prove the applicability of Theorem~\ref{thm:1.7} to the following three families of knots discussed in the introduction:

\begin{enumerate}
\item $K$ is any L-space knot, such as a torus knot or algebraic knot, or any linear combination of such knots of the same sign/handedness;
\item $K$ is any thin knot with $\tau(K) \neq 0$, such as a (quasi-)alternating knot of non-zero signature;
\item $K$ is any linear combination of genus one knots such that the overall connected sum satisfies $\tau(K) \neq 0$.
\end{enumerate}
The first and second families are immediate from the discussion of this section; note that if $K$ is a positive L-space knot, then $V_0(nK) \geq \left[ n/2 \right]$ by Lemma~\ref{lem:5.2}. Thus the only nontrivial case is the third claim.

Using the concordance invariant $\nu^+$ introduced by Hom and Wu in \cite{HW}, Hom \cite{Homsurvey} and Kim and Park \cite{KP} defined an equivalence relation on the set of knot Floer complexes called \textit{$\nu^+$-equivalence}. The $\nu^+$-equivalence class of a knot is a concordance invariant which is well-defined with respect to connected sums/tensor products. Moreover, the numerical invariants $V_0(K)$ and $\tau(K)$ may both be computed from the $\nu^+$-equivalence class of $\CFK(K)$. (Note that $\nu^+$-equivalence is the same as the stable equivalence of \cite{Homsurvey}.) In \cite{Sato}, Sato determined the $\nu^+$-equivalence class of all genus one knots. According to \cite[Theorem 1.2]{Sato}, if $K$ is a genus one knot, then
\[
\CFK^\infty(K) \sim_{\nu^+} 
\begin{cases}
                        \CFK^\infty(T_{2,3}) \quad &\text{if $\tau(K) = 1$} \\
                        \CFK^\infty(U) \quad &\text{if $\tau(K) = 0$} \\
                        \CFK^\infty(-T_{2,3}) \quad &\text{if $\tau(K) = -1$}.
\end{cases}
\]
This implies that if $K$ is a linear combination of genus one knots, then $\CFK^\infty(K) \sim_{\nu^+} cT_{2,3}$, where $c = \tau(K)$. Hence for the purposes of calculation $V_0(nK)$, we may assume that $K$ is $cT_{2,3}$. But $cT_{2,3}$ is a thin knot, for which we have already established the desired claim.

\section{Whitehead doubles}\label{sec:6}
We close by proving the applications to Whitehead doubles mentioned in the introduction. Firstly, note that Corollary~\ref{cor:A} is an immediate consequence of Theorem~\ref{thm:1.10}:

\begin{proof}[Proof of Corollary~\ref{cor:A}]
We have $\tau(nT_{2,2k+1}) = nk$; applying Theorem~\ref{thm:1.10} gives the claim.
\end{proof}

We thus turn to Corollary~\ref{cor:B}. This is a particularly simple case of the setup of Section~\ref{sec:3}, in the sense that  
\[
\Sigma_2(D(K)) = S^3_{1/2}(K \# K^r)
\]
for any $K$. The same application of Theorem~\ref{thm:2.13} as in the proof of Theorem~\ref{thm:3.1} shows that up to an overall grading shift, the $\iota$-complex $C$ of $\Sigma_2(D(K))$ thus satisfies 
\[
C^\vee \leq X_{V_0(K \# K^r)}.
\]
Note that here, $[p/(2q)] = [1/4] = 0$. In contrast to the proof of Theorem~\ref{thm:1.9}, we forego the negative-definite cobordism of Section~\ref{sec:3.2} and instead utilize $V_0(K \# K^r)$ directly.

\begin{proof}[Proof of Corollary~\ref{cor:B}]
This is a special case of Theorem~\ref{thm:3.1}. Let $\{K_n\}_{n \in \N}$ be any family of companion knots and let $C_n^\vee$ be the $\iota$-complex of $\Sigma_2(D(K_n))$. Up to overall grading shift, we have
\[
C_n^\vee \leq X_{V_0(K_n \# K_n^r)}
\]
for each $n$. If $V_0(K_n \# K_n^r) \rightarrow \infty$, then by Theorem~\ref{thm:2.12} there exists an infinite linearly independent subset of $\{D(K_n)\}_{n \in \N}$. It thus suffices to find a family of companion knots with $V_0(K_n \# K_n^r) \rightarrow \infty$ but each $\tau(K_n) \leq 0$. This is provided in Example~\ref{ex:6.3} below.
\end{proof}

The above discussion can also be used to answer a conjecture of the second author and Pinz\'on-Caicedo \cite{HPC}, who asked whether there is a knot $K$ such that the Whitehead doubles $\Wh(K)$ and $\Wh(-K)$ are non-zero in concordance. We prove the following general condition:

\begin{corollary}\label{cor:6.1}
Let $K$ be any knot with $V_0(K \# K^r) > 0$ and $\tau(K) < 0$. Then $\Wh(K)$ and $\Wh(-K)$ are linearly independent.
\end{corollary}

\begin{proof}[Proof of Corollary~\ref{cor:6.1}]
If $V_0(K \# K^r) > 0$, then a trivial application of the proof of Theorem~\ref{thm:2.12} shows that the $\iota$-complex $C$ of $\Sigma_2(D(K))$ is nontorsion. Hence $D(K)$ is nontorsion in $\cC$. Since $\tau(K) < 0$, \cite[Theorem 1.4]{Hedden} implies that $\tau(D(K)) = 0$ and $\tau(D(-K)) = 1$. As $\tau$ is a homomorphism, this shows that $D(-K)$ is also nontorsion in $\cC$ and that it is linearly independent with $D(K)$.
\end{proof}

We now give several infinite families of knots for which $V_0(K \# K^r) > 0$ and $\tau(K) \leq 0$. This condition turns out to be fairly common; we give a flexible recipe for constructing a wide class of examples below. Let $A$ and $B$ be a pair of knots such that
\[
V_0(2A) > V_0(2B) \quad \text{and} \quad \tau(A) \leq \tau(B).
\]
Then we claim that $K = A \# -B$ is a knot with the desired properties. To see this, first note that since (non-involutive) knot Floer homology is insensitive to orientation reversal, we may replace $K \# K^r$ with $2K$. Subadditivity of $V_0$ then gives the lower bound
\[
0 < V_0(2A) - V_0(2B) \leq V_0(2A \#- 2B) = V_0(2K),
\]
while $\tau(K) = \tau(A) - \tau(B) \leq 0$. The advantage of the ansatz $K = A \# -B$ is that if $A$ and $B$ are sums of positive L-space knots (or are locally equivalent to such sums), then the quantities $V_0(2A)$ and $V_0(2B)$ are easily computed via the algorithm of \cite[Proposition 5.1]{Borodzik} (described in the proof of Lemma~\ref{lem:5.2}).  

\begin{figure}[h!]
\center
\includegraphics[scale=1.2]{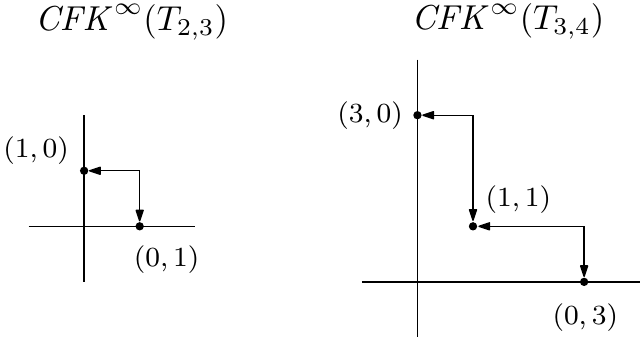}
\caption{The knot Floer complexes of $T_{2,3}$ and $T_{3,4}$.}\label{fig:6.1}
\end{figure}

\begin{example}\label{ex:6.2}
As a first example, we illustrate the above procedure for the example $K = 5T_{2,3} \# - 2T_{3,4}$ of Lewark and Zibrowius \cite[Corollary 1.13]{LZ}. In this case, $A = 5T_{2,3}$ and $B = 2T_{3,4}$. The knot Floer complexes of $T_{2,3}$ and $T_{3,4}$ are displayed in Figure~\ref{fig:6.1}. Applying the algorithm of \cite[Proposition 5.1]{Borodzik} easily shows that $V_0(2A) = V_0(10T_{2,3}) = 5$. (Alternatively, one can use the fact that $10T_{2,3}$ is thin.) Similarly, the algorithm of \cite{Borodzik} shows that $V_0(2B) = V_0(4T_{3,4}) = 4$. On the other hand, $\tau(A) = \tau(5T_{2,3}) = 5$ while $\tau(B) = \tau(2T_{3,4}) = 6$.
\end{example}

Many similar examples can be constructed by forming the difference of sums of torus knots in the style of Example~\ref{ex:6.2}; for instance, $\{n(n-1)T_{2,3} \# -2 T_{n, n +1}\}_{n \geq 2}$ or $\{T_{2, 2(n^2 - n -1) + 1} \# - T_{n, n+1}\}_{n \geq 2}$. We also provide an example of an infinite family where the companion knots are topologically slice:

\begin{figure}[h!]
\center
\includegraphics[scale=1.2]{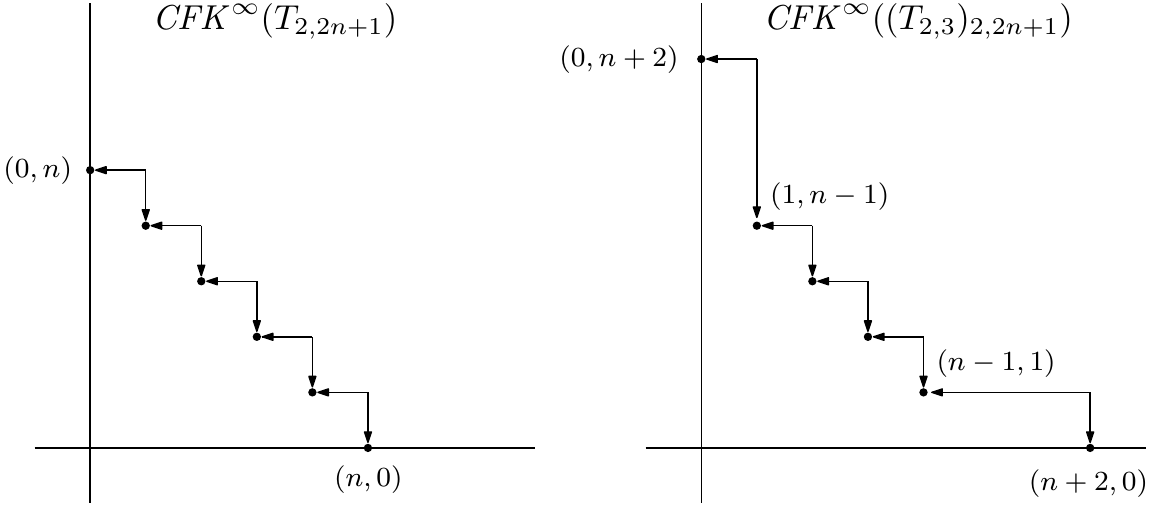}
\caption{The knot Floer complexes of $T_{2, 2n+1}$ and $(T_{2,3})_{2, 2n+1}$. The longer vertical arrows on the right have length three, while all other arrows have length one; note that $n \geq 2$. Here we have drawn the case of $n = 5$. }\label{fig:6.2}
\end{figure}

\begin{example}\label{ex:6.3}
Let $D = D(T_{2,3})$ be the Whitehead double of $T_{2,3}$. For any $n \geq 2$ and non-negative integers $p$ and $q$, let
\[
K_{n, p, q} = p(T_{2, 2n+1} - D_{2, 2n+1}) + q D.
\]
We claim that $K_{n, p, q}$ is topologically slice. To see this, note that $D$ is topologically concordant to the unknot; hence its $(2, 2n+1)$-cable $D_{2, 2n+1}$ is topologically concordant to $T_{2, 2n+1}$. Thus both of the summands $p(T_{2, 2n+1} - D_{2, 2n+1})$ and $qD$ above are topologically slice.

We fit the family $K_{n, p, q}$ into our ansatz by setting
\[
A = pT_{2, 2n+1} + qD \quad \text{and} \quad B = pD_{2,2n+1}.
\]
Although the knots $D$ and $D_{2,2n+1}$ are somewhat complicated, their knot Floer complexes are locally equivalent to L-space complexes. Indeed, \cite[Theorem 1.2]{Sato} implies that for the purposes of computing $V_0$, we may replace $D$ with $T_{2,3}$, both in $A$ and in the cable $D_{2,2n+1}$. (For an explanation of the latter, see \cite[Proposition 4]{Homconcordance}.) The complexes of $T_{2, 2n+1}$ and $(T_{2,3})_{2, 2n+1}$ are displayed in Figure~\ref{fig:6.2}. To compute the complex of $(T_{2,3})_{2, 2n+1}$, we use \cite[Theorem 1.10]{hedden2009knot}, which implies that $(T_{2,3})_{2, 2n+1}$ is an L-space knot. The computation then follows from the behavior of the Alexander polynomial under cabling.

Applying the algorithm of \cite[Proposition 5.1]{Borodzik} shows that
\[
V_0(2A) = V_0(2pT_{2, 2n+1} + 2qT_{2,3}) = pn + q \quad \text{and} \quad V_0(2B) = V_0(2p(T_{2,3})_{2, 2n+1}) = pn.
\]
Hence $V_0(2A) - V_0(2B) = q$. On the other hand,
\[
\tau(A) = \tau(pT_{2, 2n+1} + qT_{2,3}) = pn + q \quad \text{and} \quad \tau(B) = \tau(p(T_{2,3})_{2, 2n+1}) = p(n + 2),
\]
showing that $\tau(K_1) - \tau(K_2) = q - 2p$. We may thus choose any infinite family of $(p, q)$ with $0 < q \leq 2p$ and $q$ unbounded to guarantee an infinite linearly independent subset and complete the proof of Corollary~\ref{cor:B}. If strict inequality $q < 2p$ holds, then $D(K_{n, p, q})$ and $D(-K_{n, p, q})$ are moreover linearly independent in each case. Note that if $K_{n, p, q}$ satisfies the above properties, then any positive multiple of $K_{n, p, q}$ does as well; hence we obtain rank-expansion along such $K_{n, p, q}$.
\end{example}

\bibliographystyle{amsalpha}
\bibliography{bib}

\newcommand{\etalchar}[1]{$^{#1}$}
\providecommand{\bysame}{\leavevmode\hbox to3em{\hrulefill}\thinspace}
\providecommand{\MR}{\relax\ifhmode\unskip\space\fi MR }
\providecommand{\MRhref}[2]{%
  \href{http://www.ams.org/mathscinet-getitem?mr=#1}{#2}
}
\providecommand{\href}[2]{#2}
\begin{thebibliography}{LMPC22}

\bibitem[BL14]{Borodzik}
Maciej Borodzik and Charles Livingston, \emph{Heegaard {F}loer homology and
  rational cuspidal curves}, Forum Math. Sigma \textbf{2} (2014), Paper No.
  e28, 23. \MR{3347955}

\bibitem[CDR14]{injectivity}
Tim~D. Cochran, Christopher~W. Davis, and Arunima Ray, \emph{Injectivity of
  satellite operators in knot concordance}, J. Topol. \textbf{7} (2014), no.~4,
  948--964. \MR{3286894}

\bibitem[CH15]{CochranHorn}
Tim~D. Cochran and Peter~D. Horn, \emph{Structure in the bipolar filtration of
  topologically slice knots}, Algebr. Geom. Topol. \textbf{15} (2015), no.~1,
  415--428. \MR{3325742}

\bibitem[{Che}19]{Chen11patterns}
Wenzhao {Chen}, \emph{{Knot Floer homology of satellite knots with
  (1,1)-patterns}}, arXiv e-prints (2019), arXiv:1912.07914.

\bibitem[CHH13]{CochranHarveyHorn}
Tim~D. Cochran, Shelly Harvey, and Peter Horn, \emph{Filtering smooth
  concordance classes of topologically slice knots}, Geom. Topol. \textbf{17}
  (2013), no.~4, 2103--2162. \MR{3109864}

\bibitem[CHL11a]{2torsionsolvable}
Tim~D. Cochran, Shelly Harvey, and Constance Leidy, \emph{2-torsion in the
  {$n$}-solvable filtration of the knot concordance group}, Proc. Lond. Math.
  Soc. (3) \textbf{102} (2011), no.~2, 257--290. \MR{2769115}

\bibitem[CHL11b]{fractal}
\bysame, \emph{Primary decomposition and the fractal nature of knot
  concordance}, Math. Ann. \textbf{351} (2011), no.~2, 443--508. \MR{2836668}

\bibitem[CHP17]{grope}
Tim~D. Cochran, Shelly Harvey, and Mark Powell, \emph{Grope metrics on the knot
  concordance set}, J. Topol. \textbf{10} (2017), no.~3, 669--699. \MR{3665407}

\bibitem[CO93]{cochranorr}
Tim~D. Cochran and Kent~E. Orr, \emph{Not all links are concordant to boundary
  links}, Ann. of Math. (2) \textbf{138} (1993), no.~3, 519--554. \MR{1247992}

\bibitem[Con70]{Conway}
J.~H. Conway, \emph{An enumeration of knots and links, and some of their
  algebraic properties}, Computational {P}roblems in {A}bstract {A}lgebra
  ({P}roc. {C}onf., {O}xford, 1967), Pergamon, Oxford, 1970, pp.~329--358.
  \MR{0258014}

\bibitem[CT07]{cochranteichner}
Tim~D. Cochran and Peter Teichner, \emph{Knot concordance and von {N}eumann
  {$\rho$}-invariants}, Duke Math. J. \textbf{137} (2007), no.~2, 337--379.
  \MR{2309149}

\bibitem[Dae20]{daemi2020chern}
Aliakbar Daemi, \emph{{C}hern--{S}imons functional and the homology cobordism
  group}, Duke Mathematical Journal \textbf{169} (2020), no.~15, 2827--2886.

\bibitem[DHM20]{DHM}
Irving Dai, Matthew Hedden, and Abhishek Mallick, \emph{Corks, involutions, and
  {H}eegaard {F}loer homology}, 2020, preprint, arXiv:2002.02326.

\bibitem[DHST18]{DHSTcobord}
Irving Dai, Jennifer Hom, Matthew Stoffregen, and Linh Truong, \emph{An
  infinite-rank summand of the homology cobordism group}, 2018, preprint,
  arXiv:1810.06145.

\bibitem[DIS{\etalchar{+}}22]{DISST}
Aliakbar {Daemi}, Hayato {Imori}, Kouki {Sato}, Christopher {Scaduto}, and
  Masaki {Taniguchi}, \emph{{Instantons, special cycles, and knot
  concordance}}, arXiv e-prints (2022), arXiv:2209.05400.

\bibitem[DM19]{DaiManolescu}
Irving Dai and Ciprian Manolescu, \emph{Involutive {H}eegaard {F}loer homology
  and plumbed three-manifolds}, J. Inst. Math. Jussieu \textbf{18} (2019),
  no.~6, 1115--1155. \MR{4021102}

\bibitem[DR16]{davisray}
Christopher~W. Davis and Arunima Ray, \emph{Satellite operators as group
  actions on knot concordance}, Algebr. Geom. Topol. \textbf{16} (2016), no.~2,
  945--969. \MR{3493412}

\bibitem[DS19]{DaiStoffregen}
Irving Dai and Matthew Stoffregen, \emph{On homology cobordism and local
  equivalence between plumbed manifolds}, Geom. Topol. \textbf{23} (2019),
  no.~2, 865--924. \MR{3939054}

\bibitem[Gor09]{Gordon}
Cameron Gordon, \emph{Dehn surgery and 3-manifolds}, Low dimensional topology,
  IAS/Park City Math. Ser., vol.~15, Amer. Math. Soc., Providence, RI, 2009,
  pp.~21--71. \MR{2503492}

\bibitem[Hed07]{Hedden}
Matthew Hedden, \emph{Knot {F}loer homology of {W}hitehead doubles}, Geom.
  Topol. \textbf{11} (2007), 2277--2338. \MR{2372849}

\bibitem[Hed09]{hedden2009knot}
\bysame, \emph{On knot {F}loer homology and cabling: 2}, International
  Mathematics Research Notices \textbf{2009} (2009), no.~12, 2248--2274.

\bibitem[HHL21]{HHL}
Kristen Hendricks, Jennifer Hom, and Tye Lidman, \emph{Applications of
  involutive {H}eegaard {F}loer homology}, J. Inst. Math. Jussieu \textbf{20}
  (2021), no.~1, 187--224. \MR{4205781}

\bibitem[HHSZ20]{HHSZ}
Kristen Hendricks, Jennifer Hom, Matthew Stoffregen, and Ian Zemke,
  \emph{Surgery exact triangles in involutive {H}eegaard {F}loer homology},
  2020, preprint, arXiv:2011.00113.

\bibitem[HK12]{HK}
Matthew Hedden and Paul Kirk, \emph{Instantons, concordance, and {W}hitehead
  doubling}, J. Differential Geom. \textbf{91} (2012), no.~2, 281--319.
  \MR{2971290}

\bibitem[HKL16]{HKL}
Matthew Hedden, Se-Goo Kim, and Charles Livingston, \emph{Topologically slice
  knots of smooth concordance order two}, J. Differential Geom. \textbf{102}
  (2016), no.~3, 353--393. \MR{3466802}

\bibitem[HLR12]{HLR}
Matthew Hedden, Charles Livingston, and Daniel Ruberman, \emph{Topologically
  slice knots with nontrivial {A}lexander polynomial}, Adv. Math. \textbf{231}
  (2012), no.~2, 913--939. \MR{2955197}

\bibitem[HM17]{HM}
Kristen Hendricks and Ciprian Manolescu, \emph{Involutive {H}eegaard {F}loer
  homology}, Duke Math. J. \textbf{166} (2017), no.~7, 1211--1299.

\bibitem[HMZ18]{HMZ}
Kristen Hendricks, Ciprian Manolescu, and Ian Zemke, \emph{A connected sum
  formula for involutive {H}eegaard {F}loer homology}, Selecta Math. (N.S.)
  \textbf{24} (2018), no.~2, 1183--1245.

\bibitem[Hom14]{Homconcordance}
Jennifer Hom, \emph{The knot {F}loer complex and the smooth concordance group},
  Comment. Math. Helv. \textbf{89} (2014), no.~3, 537--570.

\bibitem[Hom15]{Hominfiniterank}
\bysame, \emph{An infinite-rank summand of topologically slice knots}, Geom.
  Topol. \textbf{19} (2015), no.~2, 1063--1110.

\bibitem[Hom17a]{Hom2017survey}
\bysame, \emph{A survey on {H}eegaard {F}loer homology and concordance},
  Journal of Knot Theory and Its Ramifications \textbf{26} (2017), no.~02,
  1740015.

\bibitem[Hom17b]{Homsurvey}
\bysame, \emph{A survey on {H}eegaard {F}loer homology and concordance}, J.
  Knot Theory Ramifications \textbf{26} (2017), no.~2, 1740015, 24.

\bibitem[HPC21]{HPC}
Matthew Hedden and Juanita Pinz\'{o}n-Caicedo, \emph{Satellites of infinite
  rank in the smooth concordance group}, Invent. Math. \textbf{225} (2021),
  no.~1, 131--157. \MR{4270665}

\bibitem[HW16]{HW}
Jennifer Hom and Zhongtao Wu, \emph{Four-ball genus bounds and a refinement of
  the {O}zv\'{a}th-{S}zab\'{o} tau invariant}, J. Symplectic Geom. \textbf{14}
  (2016), no.~1, 305--323. \MR{3523259}

\bibitem[KP18]{KP}
Min~Hoon Kim and Kyungbae Park, \emph{An infinite-rank summand of knots with
  trivial {A}lexander polynomial}, J. Symplectic Geom. \textbf{16} (2018),
  no.~6, 1749--1771. \MR{3934241}

\bibitem[Lev16]{Levinenonsurj}
Adam~Simon Levine, \emph{Nonsurjective satellite operators and piecewise-linear
  concordance}, Forum Math. Sigma \textbf{4} (2016), e34, 47.

\bibitem[Lit84]{MR780587}
R.~A. Litherland, \emph{Cobordism of satellite knots}, Four-manifold theory
  ({D}urham, {N}.{H}., 1982), Contemp. Math., vol.~35, Amer. Math. Soc.,
  Providence, RI, 1984, pp.~327--362. \MR{780587}

\bibitem[Liv83]{reverses}
Charles Livingston, \emph{Knots which are not concordant to their reverses},
  Quart. J. Math. Oxford Ser. (2) \textbf{34} (1983), no.~135, 323--328.
  \MR{711524}

\bibitem[Liv90]{livingston-boundary}
\bysame, \emph{Links not concordant to boundary links}, Proc. Amer. Math. Soc.
  \textbf{110} (1990), no.~4, 1129--1131. \MR{1031670}

\bibitem[Liv01]{amphicheiral}
\bysame, \emph{Infinite order amphicheiral knots}, Algebr. Geom. Topol.
  \textbf{1} (2001), 231--241. \MR{1823500}

\bibitem[LMPC22]{lidman2022linking}
Tye Lidman, Allison~N Miller, and Juanita Pinz{\'o}n-Caicedo, \emph{Linking
  number obstructions to satellite homomorphisms}, arXiv preprint
  arXiv:2207.14198 (2022).

\bibitem[LN06]{LivingstonNaik}
Charles Livingston and Swatee Naik, \emph{Ozsv\'{a}th-{S}zab\'{o} and
  {R}asmussen invariants of doubled knots}, Algebr. Geom. Topol. \textbf{6}
  (2006), 651--657. \MR{2240910}

\bibitem[LZ22]{LZ}
Lukas Lewark and Claudius Zibrowius, \emph{Rasmussen invariants of whitehead
  doubles and other satellites}, 2022, preprint, arXiv:2208.13612.

\bibitem[{Mil}19]{MillerHomo}
Allison~N. {Miller}, \emph{{Homomorphism obstructions for satellite maps}},
  arXiv e-prints (2019), arXiv:1910.03461.

\bibitem[MZ21]{MZ}
Duncan McCoy and Raphael Zentner, \emph{The {M}ontesinos trick for proper
  rational tangle replacement}, 2021, preprint, arXiv:2110.15106.

\bibitem[NST19]{NST}
Yuta Nozaki, Kouki Sato, and Masaki Taniguchi, \emph{Filtered instanton floer
  homology and the homology cobordism group}, 2019, preprint, arXiv:1905.04001.

\bibitem[NW15]{NiWu}
Yi~Ni and Zhongtao Wu, \emph{Cosmetic surgeries on knots in {$S^3$}}, J. Reine
  Angew. Math. \textbf{706} (2015), 1--17.

\bibitem[OS04a]{OSknots}
Peter Ozsv{\'a}th and Zolt{\'a}n Szab{\'o}, \emph{Holomorphic disks and knot
  invariants}, Adv. Math. \textbf{186} (2004), no.~1, 58--116.

\bibitem[OS04b]{OS3manifolds2}
\bysame, \emph{Holomorphic disks and three-manifold invariants: properties and
  applications}, Ann. of Math. (2) \textbf{159} (2004), no.~3, 1159--1245.

\bibitem[OS04c]{OS3manifolds1}
\bysame, \emph{Holomorphic disks and topological invariants for closed
  three-manifolds}, Ann. of Math. (2) \textbf{159} (2004), no.~3, 1027--1158.

\bibitem[OS05]{OSlens}
\bysame, \emph{On knot {F}loer homology and lens space surgeries}, Topology
  \textbf{44} (2005), no.~6, 1281--1300.

\bibitem[PC17]{PC}
Juanita Pinz\'{o}n-Caicedo, \emph{Independence of satellites of torus knots in
  the smooth concordance group}, Geom. Topol. \textbf{21} (2017), no.~6,
  3191--3211. \MR{3692965}

\bibitem[Ras03]{Rasmussen}
Jacob~Andrew Rasmussen, \emph{Floer homology and knot complements}, ProQuest
  LLC, Ann Arbor, MI, 2003, Thesis (Ph.D.)--Harvard University. \MR{2704683}

\bibitem[Ras04]{Rasmussen_lens}
Jacob Rasmussen, \emph{Lens space surgeries and a conjecture of {G}oda and
  {T}eragaito}, Geom. Topol. \textbf{8} (2004), 1013--1031. \MR{2087076}

\bibitem[Sat19]{Sato}
Kouki Sato, \emph{The $\nu^+$-equivalence classes of genus one knots}, 2019,
  preprint, arXiv:1907.09116.

\bibitem[Sav02]{SavelievInvariants}
Nikolai Saveliev, \emph{Invariants for homology 3-spheres}, vol. 140, 01 2002.

\bibitem[Son17]{Song}
Hyun-Jong Song, \emph{Two dimensional arrays for {A}lexander polynomials of
  torus knots}, Commun. Korean Math. Soc. \textbf{32} (2017), no.~1, 193--200.
  \MR{3608491}

\bibitem[Wal69]{Waldhausen}
Friedhelm Waldhausen, \emph{\"{U}ber {I}nvolutionen der {$3$}-{S}ph\"{a}re},
  Topology \textbf{8} (1969), 81--91. \MR{236916}

\end{thebibliography}

\end{document}